\documentclass[10pt]{article}
\usepackage[pagewise,displaymath,mathlines]{lineno}

\usepackage{graphicx}
\usepackage{epstopdf}
\usepackage{subfigure}
\usepackage{color}
\usepackage{amsmath,amsfonts}
\usepackage{bbm}
\usepackage[numbers,sort&compress]{natbib}
\usepackage[colorlinks,
            linkcolor=red,
            anchorcolor=blue,
            citecolor=green
            ]{hyperref}

\numberwithin{equation}{section}

\usepackage{verbatim}
\usepackage{latexsym}
\usepackage[utf8]{inputenc} 
\usepackage{amssymb}
\usepackage{mathrsfs}
\usepackage{graphicx}
\usepackage[noblocks]{authblk}
\usepackage{caption}
\date{ }

\begin{document}
\newcommand\tcb{\textcolor{blue}}
\newcommand\tcr{\textcolor{red}}

\newcommand{\E}{\boldsymbol{E}}
\newcommand{\G}{\mathcal{G}_n}
\newtheorem{theorem}{Theorem}[section]
\newtheorem{lemma}[theorem]{Lemma}
\newtheorem{coro}[theorem]{Corollary}
\newtheorem{defn}[theorem]{Definition}
\newtheorem{assumption}[theorem]{Assumption}
\newtheorem{example}[theorem]{Example}
\newtheorem{prop}[theorem]{Proposition}
\newtheorem{remark}[theorem]{Remark}

\newcommand{\upcite}[1]{\textsuperscript{\textsuperscript{\cite{#1}}}}
\newcommand\tq{{\scriptstyle{3\over 4 }\scriptstyle}}
\newcommand\qua{{\scriptstyle{1\over 4 }\scriptstyle}}
\newcommand\hf{{\textstyle{1\over 2 }\displaystyle}}
\newcommand\athird{{\scriptstyle{1\over 3 }\scriptstyle}}
\newcommand\hhf{{\scriptstyle{1\over 2 }\scriptstyle}}

\newcommand{\eproof}{\indent\vrule height6pt width4pt depth1pt\hfil\par\medbreak}
\newcommand{\RR}{\mathbb{R}}
\makeatletter \@addtoreset{equation}{section}
\renewcommand{\thesection}{\arabic{section}}
\renewcommand{\theequation}{\thesection.\arabic{equation}}

\def\a{\alpha} \def\g{\gamma}
\def\e{\varepsilon} \def\z{\zeta} \def\y{\eta} \def\o{\theta}
\def\vo{\vartheta} \def\k{\kappa} \def\l{\lambda} \def\m{\mu} \def\n{\nu}
\def\x{\xi}  \def\r{\rho} \def\s{\sigma}
\def\p{\varphi} \def\f{\varphi}   \def\w{\omega}
\def\q{\surd} \def\i{\bot} \def\h{\forall} \def\j{\emptyset}

\def\be{\beta} \def\de{\delta} \def\up{\upsilon} \def\eq{\equiv}
\def\ve{\vee} \def\we{\wedge}

\def\F{{\cal F}}
\def\T{\tau}   \def\D{\Delta} \def\O{\Theta} \def\L{\Lambda}
\def\X{\Xi} \def\S{\Sigma} \def\W{\Omega}
\def\M{\partial} \def\N{\nabla} \def\Ex{\exists} \def\K{\times}
\def\V{\bigvee} \def\U{\bigwedge}

\def\1{\oslash} \def\2{\oplus} \def\3{\otimes} \def\4{\ominus}
\def\5{\circ} \def\6{\odot} \def\7{\backslash} \def\8{\infty}
\def\9{\bigcap} \def\0{\bigcup} \def\+{\pm} \def\-{\mp}
\def\[{\langle} \def\]{\rangle}

\def\proof{\noindent{\it Proof. }}
\def\tl{\tilde}
\def\trace{\hbox{\rm trace}}
\def\diag{\hbox{\rm diag}}
\def\for{\quad\hbox{for }}
\def\refer{\hangindent=0.3in\hangafter=1}

\newcommand\wD{\widehat{\D}}

\thispagestyle{empty}
\title{ \bf  An explicit Euler method for McKean–Vlasov SDEs  driven by fractional Brownian motion}
\author{

Jie He\textsuperscript{a},\quad
Shuaibin Gao\textsuperscript{a},\quad
Weijun Zhan\textsuperscript{b},\quad
Qian Guo\textsuperscript{a}\thanks{Corresponding author, Email: qguo@shnu.edu.cn}\quad
\\
\textsuperscript{a}
Department of Mathematics, Shanghai Normal University, Shanghai 200234, China\\
\textsuperscript{b}Department of Mathematics, Anhui Normal University, Wuhu 241000, China\\
}
\maketitle

\begin{abstract}
In this paper,  we establish the theory of chaos propagation and propose an Euler-Maruyama scheme for McKean–Vlasov stochastic differential equations driven by fractional Brownian motion with Hurst exponent $H\in (0,1)$. Meanwhile, upper bounds for errors in the Euler method is obtained. A numerical example is demonstrated to verify the theoretical results.

\end{abstract}

\section{Introduction}

The pioneering work of McKean–Vlasov stochastic differential equations (SDEs) has been done by McKean in  \cite{mckean1966,mckean1967,mckean1975} connected with a mathematical foundation of the Boltzmann equation. Due to their widespread applications in many fields, McKean–Vlasov SDEs have been researched by many scholars. In \cite{huang2019}, existence and uniqueness are proved for distribution-dependent SDEs with non-degenerate noise under integrability conditions on distribution-dependent coefficients.  Some theories about McKean–Vlasov SDEs were investigated, including 
ergodicity \cite{eberle2019}, Harnack inequality \cite{wang2018}, and the Bismut formula \cite{fan2022,ren2019}. And the integration by parts formulae on Wiener space for solutions of the SDEs with general McKean–Vlasov interaction was derived in \cite{crisan2018}. In \cite{buckdahn2017}, Buckdahn et al. characterized the function on the coefficients of the stochastic differential equation under appropriate regularity conditions as the unique classical solution of a nonlocal partial differential equation of mean-field type.  A complete probabilistic analysis of a large class of stochastic differential games with mean field interactions was provided in \cite{carmona2013}.  

It is well known that the explicit solutions to McKean–Vlasov SDEs are difficult to be shown. Hence, the numerical methods for McKean–Vlasov SDEs driven by standard Brownian motion are studied by many scholars \cite{antonelli2002,bao2021,bossy1996,bossy1997,dos2022}. Moreover, it should be noted that SDEs driven by fractional Brownian motion (fBm) have wider applications \cite{biagini2008,mishura2008,norros1995}. On the other hand, the numerical methods for SDEs driven by fBm have attracted increasing interest; see \cite{hong2020,hu2021,li2021,wang2022,zhang2021} for example. Galeati et al. \cite{galeati2022} examined the distribution-dependent stochastic differential equations with erratic, potentially distributional drift, driven by an additive fBm of Hurst parameter $H\in (0, 1)$, and they established strong well-posedness under a variety of assumptions on the drift. To our knowledge, the numerical method for McKean–Vlasov SDEs driven by fBm has not been discussed yet.
As we know, propagation of chaos plays a key role to approximate the Mckean-Vlasov SDEs.
This paper aims at establishing the theory about propagation of chaos and the strong convergence rate in $L^{p}$ sense of EM method for  McKean–Vlasov SDEs  driven by fBm under the globally Lipschitz condition.

In this paper, we consider the following $d$-dimensional  McKean–Vlasov SDEs driven by fBm of the form
\begin{equation}\label{McKean}
\mathrm{d}X_{t}=b\left(X_{t}, \mathcal{L}\left(X_{t}\right)\right) \mathrm{d}t+ \sigma\left(\mathcal{L}\left(X_{t}\right)\right) \mathrm{d} B^{H}_{t},
\end{equation}
where the coefficients $b:\mathbb{R}^{d} \times \mathcal{P}_{\theta}\left(\mathbb{R}^{d}\right) \rightarrow \mathbb{R}^{d}, \sigma:\mathcal{P}_{\theta}\left(\mathbb{R}^{d}\right) \rightarrow \mathbb{R}^{d}\otimes\mathbb{R}^{d}$. Here, the initial value $X_{0} \in L^{p}\left(\Omega \rightarrow \mathbb{R}^{d}, \mathcal{F}_{0}, \mathbb{P}\right)$ with $p \geq \theta\geq 2$,
 and  $B^{H}_{ t}$ is a $d$-dimensional fBm with Hurst parameter $H\in (0,1)$. As we know, the covariance of $B^{H}_{t}$ is $$R_{H}\left(t,s\right)=\boldsymbol{E}\left(B^{H}_{t}B^{H}_{s}\right)=\frac{1}{2}\left(t^{2H}+s^{2H}-|t-s|^{2H}\right), \quad\forall t, s \in[0,T].$$ The fBm $\{B^{H}_{t}\}_{t\geq0}$  corresponds to a standard Brownian motion when $H = 1/2$. Further, the fBm $\{B^{H}_{t}\}_{t\geq0}$ is not a semi-martingale or a Markov process unless $H = 1/2$.  Therefore, when working with the fBm $\{B^{H}_{t}\}_{t\geq0}$, many of the powerful features are unavailable.  
We rewrite \eqref{McKean} to the system of noninteracting particles
\begin{equation}\label{McKeanN}
X_{t}^{i}=X_{0}^{i}+\int_{0}^{t}b\left(X_{s}^{i}, \mathcal{L}_{X_{s}^{i}}\right) \mathrm{d} s+\int_{0}^{t} \sigma\left( \mathcal{L}_{X_{s}^{i}}\right) \mathrm{d} B^{H,i}_{s},\quad i=1,\ldots,N,
\end{equation}
 where $\mathcal{L}_{X_{t}^{i}}$ denotes the law of the process $X^{i}$ at time $t$.
Compared with the standard SDEs, McKean–Vlasov SDEs provide an additional complexity, that is, it is required to approximate the law at each time step. Although there are other technologies, the most common one is the so-called interacting particle system
\begin{equation}\label{McKeanI}
X_{t}^{i,N}=X_{0}^{i}+\int_{0}^{t}b\left(X_{s}^{i,N}, \mu_{s}^{X,N}\right) \mathrm{d} s+\int_{0}^{t} \sigma\left( \mu_{s}^{X,N}\right) \mathrm{d} B^{H,i}_{s},
\end{equation}
where the empirical measures is defined by
$$
\mu_{t}^{X,N}(\cdot):=\frac{1}{N} \sum_{j=1}^{N} \delta_{X_{t}^{j,N}}(\cdot),
$$
and $\delta_{x}$ denotes the Dirac measure at point $x$.

The structure of this work is as follows. The mathematical preliminaries on the McKean–Vlasov SDEs driven by fBm are presented in Section \ref{part2}. Section \ref{part3} gives the main theorem and its proof. Numerical simulations are provided in Section \ref{part4}.

\section{Mathematical Preliminaries}\label{part2}

Throughout the article, we will always work on a finite time interval $[0, T ]$ and consider an underlying probability space $(	\Omega, \mathcal{F}, \mathbb{P})$. More precisely, $\Omega$ is the Banach space of continuous functions vanishing at 0 equipped with the supremum norm, $\mathcal{F}$ is the Borel $\sigma$-algebra and $\mathbb{P}$ is the unique probability measure on $\Omega$ such that the canonical process $\{B ^{H}_{t} \}_{t \in[0, T]}$ is a $d$-dimensional fBm with Hurst parameter $H\in (0,1)$. For any $p \geq 1$, let 
$$
\|V(x)\|_{L^{p}}:=\left(\boldsymbol{E} \left|V(x)\right|^{p}\right)^{1 / p}.
$$
We use $| \cdot |$ and $\langle\cdot, \cdot\rangle$
for the Euclidean norm and inner product, respectively, 
 and let $a\wedge b :=\min(a, b)$ and $a\vee b :=\max(a, b)$.  The notation $u\otimes v$ for $u\in \mathbb{R}^d$ and $v\in \mathbb{R}^d$ means the tensor product of $u$ and $v$. We will denote the set of all probability measures on $\mathbb{R}^{d}$ by
\begin{equation*}
 \mathcal{P}_{p}\left(\mathbb{R}^{d}\right):=\left\{\mu \in \mathcal{P}\left(\mathbb{R}^{d}\right): \int_{\mathbb{R}^{d}}|x|^{p} \mu(\mathrm{d} x)<\infty\right\}
\end{equation*}
For $\theta\ge 2$ and any $\mu, \nu\in\mathcal{P}_{\theta}(\mathbb{R}^{d})$, the $\theta$-Wasserstein distance is defined by,
 \begin{equation*}
 \mathcal{W}_{\theta}(\mu, \nu):=\left(\inf _{\pi \in \Pi(\mu, \nu)} \int_{\mathbb{R}^{d} \times \mathbb{R}^{d}}|x-y|^{\theta} \pi(\mathrm{d} x, \mathrm{d} y)\right)^{1/\theta}
 \end{equation*}
where $\Pi(\mu, \nu)$ is the set of couplings of $\mu$ and $\nu$, and $\mathcal{P}_{\theta}\left(\mathbb{R}^{d}\right)$ is a Polish space under the $\theta$-Wasserstein metric.

Let $a, b \in \mathbb{R}$ with $a<b$. For $g \in L^{1}(a, b)$ and $\alpha>0$, the left-sided fractional Riemann-Liouville integral of $g$ of order $\alpha$ on $[a, b]$ is defined as
$$
I_{a+}^{\alpha} g(x)=\frac{1}{\Gamma(\alpha)} \int_{a}^{x}(x-y)^{\alpha-1} g(y) \mathrm{d} y,
$$
and the right-sided fractional Riemann-Liouville integral of $f$ of order $\alpha$ on $[a, b]$ is defined as
$$
 I_{b-}^{\alpha} g(x)=\frac{(-1)^{-\alpha}}{\Gamma(\alpha)} \int_{x}^{b}(y-x)^{\alpha-1}  g(y)\mathrm{d} y,
$$
where $x \in(a, b)$ 
 and $\Gamma$ denotes the Gamma function. For further details about fractional integral and derivative, we refer the reader to \cite{biagini2008,samko1993}.

\begin{assumption}\label{assp}
There exists a positive constant $L$ such that
\begin{equation}
|b(x,\mu)-b(y,\nu)|\leq L\left(|x-y|+\mathcal{W}_{\theta}(\mu, \nu)\right),
\end{equation}
 for all $x,y \in \mathbb{R}^{d}$ and $\mu,\nu \in\mathcal{P}_{\theta}(\mathbb{R}^{d})$.
\begin{equation}
|\sigma(\mu)-\sigma(\nu)|\leq L\mathcal{W}_{\theta}(\mu, \nu),
\end{equation}
 for all $\mu,\nu \in\mathcal{P}_{\theta}(\mathbb{R}^{d})$. And for initial experience distribution $\delta_{0}$,
\begin{equation}
|b(0,\delta_{0})|\vee|\sigma(\delta_{0}))|\leq L.
\end{equation}
\end{assumption}
Furthermore, from Assumption \ref{assp}, there exists  a positive constant $L$ such that
\begin{equation}\label{tui1}
|b(x,\mu)|\leq L\left(1+|x|+\mathcal{W}_{\theta}(\mu,\delta_{0} )\right),
\end{equation}
and
\begin{equation}\label{tui2}
|\sigma(\mu)|\leq L\left(1+\mathcal{W}_{\theta}(\mu,\delta_{0})\right).
\end{equation}

\section{Main Result}\label{part3}

\begin{lemma}\label{existence and uniqueness}\cite{fan2022}
When $H\in(1/2,1)$ and $p>1/H$ hold, the solution of $X_{t}$ in \eqref{McKean} exists and is unique. Moreover, when  $H\in(0,1/2)$ and $\sigma(\mu)$ is independent of distribution,  the solution of $X_{t}$ in \eqref{McKean} exists and is unique.
\end{lemma}

\begin{lemma}\label{lem1}
For two  empirical measures
$\mu_{t}^{N}=\frac{1}{N}\sum_{j=1}^N\delta_{Z^{1,j,N}_{t}}$ and $\nu_{t}^{N}=\frac{1}{N}\sum_{j=1}^N \delta_{Z^{2,j,N}_{t}}$. Then,
\begin{equation*}
\boldsymbol{E}\left(\mathcal{W}_\theta^\theta(\mu_{t}^{N}, \nu_{t}^{N})\right)
\le \boldsymbol{E}\left(\frac{1}{N}\sum_{j=1}^N \left|Z_{t}^{1,j,N}-Z_{t}^{2,j, N}\right|^\theta\right).
\end{equation*}
\end{lemma}
This lemma follows from constructing a simple transport plan $\pi(dx, dy)=\frac{1}{N}\sum_{j=1}^N \delta_{Z^{1,j,N}}(dx)\otimes \delta_{Z^{2,j,N}}(dy)$.

\subsection{Case $H>1/2$}

\begin{theorem}\label{Xbound}
Let Assumption \ref{assp} holds and $q> p\geq2$, then
\begin{equation*}
  \sup_{t\in[0,T]}\E\left(|X^{i}_{t}|^{q}\right)+\sup_{t\in[0,T]}\E\left(|X^{i,N}_{t}|^{q}\right)\leq C_{q,H,T,L}\left(1+\boldsymbol{E}|X^{i}_{0}|^{q}\right),
\end{equation*}
where $C_{q,H,T,L}$ is a positive constant dependent on $q,H,T,L$.
\end{theorem}
\begin{proof}
From \eqref{McKeanI} and elementary inequality, we get
\begin{equation}\label{J}
\begin{split}
\E\left(|X^{i,N}_{t}|^{q}\right)\leq& 3^{q-1}\boldsymbol{E}\left(|X^{i}_{0}|^{q}\right)+3^{q-1}\E\left(\left|\int_{0}^{t}b\left(X^{i,N}_{s},\mu_{s}^{X,N}\right)\mathrm{d}s\right|^{q}\right)
+3^{q-1}\boldsymbol{E}\left(\left|\int_{0}^{t}\sigma\left(\mu_{s}^{X,N}\right)\mathrm{d}B^{H,i}_{s}\right|^{q}\right).
\end{split}
\end{equation}
For the second part at the right side of \eqref{J}, using the H\"{o}lder inequality, we obtain
\begin{equation*}
\begin{split}
\boldsymbol{E}\left(\left|\int_{0}^{t}b\left(X^{i,N}_{s},\mu_{s}^{X,N}\right)\mathrm{d}s\right|^{q}\right)
\leq T^{q-1}\boldsymbol{E}\left(\int_{0}^{t}\left|b\left(X^{i,N}_{s},\mu_{s}^{X,N}\right)\right|^{q}\mathrm{d}s\right).
\end{split}
\end{equation*}
Applying \eqref{tui1}, we see that
\begin{equation}\label{J1}
\begin{split}
\boldsymbol{E}\left(\left|\int_{0}^{t}b\left(X^{i,N}_{s},\mu_{s}^{X,N}\right)\mathrm{d}s\right|^{q}\right)
\leq& T^{q-1}\boldsymbol{E}\left(\int_{0}^{t}L^{q}\left(1+|X^{i,N}_{s}|+\mathcal{W}_{\theta}(\mu_{s}^{X,N},\delta_{0} )\right)^{q}\mathrm{d}s\right)
\\\leq&C_{T,q,L}+C_{T,q,L}\int_{0}^{t}\boldsymbol{E}\left(|X^{i,N}_{s}|^{q}\right)\mathrm{d}s
+C_{T,q,L}\int_{0}^{t}\boldsymbol{E}\left(\mathcal{W}_{\theta}^{q}(\mu_{s}^{X,N},\delta_{0} )\right)\mathrm{d}s.
\end{split}
\end{equation}
For the last part of the right end of \eqref{J}, apply Theorem 1.1 in \cite{memin2001}, we get
\begin{equation}\label{zhongdian}
\boldsymbol{E}\left(\left|\int_{0}^{t}\sigma\left(\mu_{s}^{X,N}\right)dB^{H,i}_{s}\right|^{q}\right)
\leq C_{q,H}\left(\int_{0}^{t}\left|\sigma\left(\mu_{s}^{X,N}\right)\right|^{1/H}\mathrm{~d}s\right)^{qH}.
\end{equation}
Using the H\"{o}lder inequality,
\begin{equation*}
\boldsymbol{E}\left(\left|\int_{0}^{t}\sigma\left(\mu_{s}^{X,N}\right)dB^{H,i}_{s}\right|^{q}\right)
\leq C_{q,H,T}\int_{0}^{t} \left|\sigma\left(\mu_{s}^{X,N}\right)\right|^{q}\mathrm{~d}s.
\end{equation*}
By \eqref{tui2},  we have
\begin{equation}\label{J2}
\begin{split}
\boldsymbol{E}\left(\left|\int_{0}^{t}\sigma\left(\mu_{s}^{X,N}\right)dB^{H,i}_{s}\right|^{q}\right)
\leq &C_{q,H,T}\int_{0}^{t} L^{q}\boldsymbol{E}\left(1+\mathcal{W}_{\theta}(\mu_{s}^{X,N},\delta_{0} )\right)^{q}\mathrm{~d}s
\\\leq &C_{q,H,T,L}+C_{q,H,T,L}\int_{0}^{t}\boldsymbol{E}\left(\mathcal{W}_{\theta}^{q}(\mu_{s}^{X,N},\delta_{0} )\right)\mathrm{~d}s.
\end{split}
\end{equation}
Through a similar proof in  \cite[Lemma 2.3]{dos2019}, for $\theta\geq2$ one can observe that
\begin{equation}\label{delta}
\mathcal{W}_{\theta}^{\theta}(\mu_{s}^{X,N},\delta_{0} )=\frac{1}{N} \sum_{j=1}^{N}\left|X^{j,N}_{s}\right|^{\theta}.
\end{equation}
Combine \eqref{J1} and \eqref{J2} into \eqref{J} and \eqref{delta}, we get
\begin{equation*}
\begin{split}
\boldsymbol{E}\left(|X^{i,N}_{t}|^{q}\right)\leq& C_{q}\boldsymbol{E}\left(|X^{i}_{0}|^{q}\right)+C_{T,q,L}+C_{T,q,L}\int_{0}^{t}\boldsymbol{E}\left(|X^{i,N}_{s}|^{q}\right)\mathrm{~d}s
\\&+C_{T,q,L}\int_{0}^{t}\boldsymbol{E}\left(\mathcal{W}_{\theta}^{q}(\mu_{s}^{X,N},\delta_{0} )\right)\mathrm{~d}s+C_{q,H,T,L}+C_{q,H,T,L}\int_{0}^{t}\boldsymbol{E}\left(\mathcal{W}_{\theta}^{q}(\mu_{s}^{X,N},\delta_{0} )\right)\mathrm{~d}s
\\\leq& C_{q,H,T,L}\left(1+\boldsymbol{E}|X^{i}_{0}|^{q}\right)+C_{T,q,L}\int_{0}^{t}\boldsymbol{E}\left(|X^{i,N}_{s}|^{q}\right)\mathrm{~d}s
+C_{T,q,L}\boldsymbol{E}\left(\left(\frac{1}{N} \sum_{j=1}^{N}\left|X^{j,N}_{s}\right|^{\theta}\right)^{q/\theta}\right).
\end{split}
\end{equation*}
For the last term, we apply the Minkowski inequality and since all $j$ are identically distributed, we have
\begin{equation*}
\boldsymbol{E}\left(\left(\frac{1}{N} \sum_{j=1}^{N}\left|X^{j,N}_{s}\right|^{\theta}\right)^{q/\theta}\right)\leq \left(\frac{1}{N} \sum_{j=1}^{N}\left\||X^{j,N}_{s}|^{\theta}\right\|_{L^{q/\theta}}\right)^{q/\theta}=\left(\frac{1}{N} \sum_{j=1}^{N}\left(\boldsymbol{E}|X^{j,N}_{s}|^{q}\right)^{\theta/q}\right)^{q/\theta}=\boldsymbol{E}\left(|X^{i,N}_{s}|^{q}\right).
\end{equation*}
Thus, we get
\begin{equation*}
\boldsymbol{E}\left(|X^{i,N}_{t}|^{q}\right)\leq C_{q,H,T,L}\left(1+\boldsymbol{E}|X^{i}_{0}|^{q}\right)+C_{q,H,T,L}\int_{0}^{t}\boldsymbol{E}\left(|X^{i,N}_{s}|^{q}\right)\mathrm{~d}s
.
\end{equation*}
 Then applying the Gr\"{o}nwall inequality, we obtain
 \begin{equation*}
\boldsymbol{E}\left(|X^{i,N}_{t}|^{q}\right)\leq C_{q,H,T,L}\left(1+\boldsymbol{E}|X^{i}_{0}|^{q}\right).
\end{equation*}
 Similarly, we can show
 \begin{equation*}
\boldsymbol{E}\left(|X^{i}_{t}|^{q}\right)\leq C_{q,H,T,L}\left(1+\boldsymbol{E}|X^{i}_{0}|^{q}\right).
\end{equation*}
The proof is complete.\eproof
\end{proof}

\begin{theorem}
[Propagation of Chaos]\label{pc}
Let Assumption \ref{assp} be satisfied. If for some $p\in [\theta, q)$, then it holds that
\begin{equation*}
\begin{split}
\sup _{i \in\{1, \ldots, N]} \sup _{t \in[0, T]} \boldsymbol{E}\left|X_{t}^{i}-X_{t}^{i,N}\right|^{p} &\leq C_{p,T,H,L,\theta}\\&
\times
\left\{\begin{array}{ll}
N^{-1 / 2}+N^{-(q-p)/q}, & \text { if } p>d/2\text{ and }q\neq2p, \\
N^{-1 / 2} \log(1+N)+N^{-(q-p)/q}, & \text { if } p=d/2\text{ and }q\neq2p,\\
N^{-p / d}+N^{-(q-p)/q}, & \text { if }p\in[2,d/2)\text{ and }q\neq d/(d-p),
\end{array}\right.
\end{split}
\end{equation*}
where the constant $C_{p,T,H,L,\theta}> 0$ depends on $p$, $T$, $H$, $L$ and $\theta$ but does not depend on N.
\end{theorem}
\begin{proof}
It follows from \eqref{McKeanN} and \eqref{McKeanI} that
\begin{equation*}
X_{t}^{i}-X_{t}^{i,N}=\int_{0}^{t}\left(b\left(X_{s}^{i}, \mathcal{L}_{X_{s}^{i}}\right)-b\left(X_{s}^{i,N}, \mu_{s}^{X,N}\right)\right) \mathrm{~d}s+\int_{0}^{t} \left(\sigma\left( \mathcal{L}_{X_{s}^{i}}\right)-\sigma\left( \mu_{s}^{X,N}\right) \right) \mathrm{~d} B^{H,i}_{s}.
\end{equation*}
Using the elementary inequality, we can show that
\begin{equation}\label{jjj}
\begin{split}
\boldsymbol{E}\left(\left|X_{t}^{i}-X_{t}^{i,N}\right|^{p}\right)\leq&2^{p-1}\boldsymbol{E}\left(\left|\int_{0}^{t}b\left(X_{s}^{i}, \mathcal{L}_{X_{s}^{i}}\right)-b\left(X_{s}^{i,N}, \mu_{s}^{X,N}\right)\mathrm{~d}s \right|^{p}\right)
\\&+2^{p-1}\boldsymbol{E}\left(\left|\int_{0}^{t}\sigma\left( \mathcal{L}_{X_{s}^{i}}\right)-\sigma\left( \mu_{s}^{X,N}\right) \mathrm{~d} B^{H,i}_{s}\right|^{p}\right).
\end{split}
\end{equation}
For  the first part on the right of \eqref{jjj}, use the H\"{o}lder inequality and Assumption \ref{assp}, we get
\begin{equation}\label{JJ1}
\begin{split}
\boldsymbol{E}&\left(\left|\int_{0}^{t}b\left(X_{s}^{i}, \mathcal{L}_{X_{s}^{i}}\right)-b\left(X_{s}^{i,N}, \mu_{s}^{X,N}\right)\mathrm{~d}s \right|^{p}\right)
\\\leq&t^{p-1}\boldsymbol{E}\left(\int_{0}^{t}\left|b\left(X_{s}^{i}, \mathcal{L}_{X_{s}^{i}}\right)-b\left(X_{s}^{i,N}, \mu_{s}^{X,N}\right) \right|^{p}\mathrm{~d}s\right)
\\\leq&t^{p-1}\boldsymbol{E}\left(\int_{0}^{t}\left|L\left(\left|X_{s}^{i}-X_{s}^{i,N}\right|+ \mathcal{W}_{\theta}(\mathcal{L}_{X_{s}^{i}},\mu_{s}^{X,N} )\right)\right|^{p}\mathrm{~d}s\right)
\\\leq& C_{p,T,L}\boldsymbol{E}\left(\int_{0}^{t}\left|X_{s}^{i}-X_{s}^{i,N}\right|^{p}\mathrm{~d}s\right)
+C_{p,T,L}\boldsymbol{E}\left(\int_{0}^{t}\mathcal{W}_{\theta}^{p}(\mathcal{L}_{X_{s}^{i}},\mu_{s}^{X,N} )\mathrm{~d}s\right).
\end{split}
\end{equation}
For  the second part on the right of \eqref{jjj}, 
 through the same technique as  \eqref{zhongdian}, we get
\begin{equation*}
\begin{split}
\boldsymbol{E}\left(\left|\int_{0}^{t}\sigma\left( \mathcal{L}_{X_{s}^{i}}\right)-\sigma\left( \mu_{s}^{X,N}\right) \mathrm{~d} B^{H,i}_{s}\right|^{p}\right)
\leq C_{p,T,H}\boldsymbol{E}\left(\left(\int_{0}^{t}|\sigma\left( \mathcal{L}_{X_{s}^{i}}\right)-\sigma\left( \mu_{s}^{X,N}\right)|^{1/H}\mathrm{~d}s\right)^{pH}\right).
\end{split}
\end{equation*}
Apply Assumption \ref{assp}, we obtain
\begin{equation}\label{BH}
\begin{split}
\boldsymbol{E}\left(\left|\int_{0}^{t}\sigma\left( \mathcal{L}_{X_{s}^{i}}\right)-\sigma\left( \mu_{s}^{X,N}\right) \mathrm{~d} B^{H,i}_{s}\right|^{p}\right)\leq C_{p,T,H,L}\int_{0}^{t}\boldsymbol{E}\left(\mathcal{W}_{\theta}^{p}(\mathcal{L}_{X_{s}^{i}},\mu_{s}^{X,N} )\right)\mathrm{~d}s.
\end{split}
\end{equation}
Combining this with \eqref{JJ1}, then
\begin{equation*}
\begin{split}
\boldsymbol{E}\left(\left|X_{t}^{i}-X_{t}^{i,N}\right|^{p}\right)\leq&C_{p,T,L}\boldsymbol{E}\left(\int_{0}^{t}\left|X_{s}^{i}-X_{s}^{i,N}\right|^{p}\mathrm{~d}s\right)
+C_{p,T,L}\int_{0}^{t}\boldsymbol{E}\left(\mathcal{W}_{\theta}^{p}(\mathcal{L}_{X_{s}^{i}},\mu_{s}^{X,N} )\right)\mathrm{~d}s
\\&+C_{p,T,H,L}\int_{0}^{t}\boldsymbol{E}\left(\mathcal{W}_{\theta}^{p}(\mathcal{L}_{X_{s}^{i}},\mu_{s}^{X,N} )\right)\mathrm{~d}s
\\\leq&C_{p,T,L}\boldsymbol{E}\left(\int_{0}^{t}\left|X_{s}^{i}-X_{s}^{i,N}\right|^{p}\mathrm{~d}s\right)
+C_{p,T,H,L}\int_{0}^{t}\boldsymbol{E}\left(\mathcal{W}_{\theta}^{p}(\mathcal{L}_{X_{s}^{i}},\mu_{s}^{X,N} )\right)\mathrm{~d}s.
\end{split}
\end{equation*}
For the part of Wasserstein distance, we note $\mu_{t}^{X}:=\frac{1}{N} \sum_{j=1}^{N} \delta_{X_{t}^{j}}$ and we obtain
\begin{equation*}
\begin{split}
\mathcal{W}_{\theta}^{p}(\mathcal{L}_{X_{s}^{i}},\mu_{s}^{X,N} )=&\left(\mathcal{W}_{\theta}^{\theta}(\mathcal{L}_{X_{s}^{i}},\mu_{s}^{X,N} )\right)^{p/\theta}
\\\leq&\left(2^{\theta-1}\mathcal{W}_{\theta}^{\theta}(\mu_{s}^{X},\mu_{s}^{X,N} )+2^{\theta-1}\mathcal{W}_{\theta}^{\theta}(\mathcal{L}_{X_{s}^{i}},\mu_{s}^{X} )\right)^{p/\theta}
\\\leq&C_{p,\theta}\mathcal{W}_{\theta}^{p}(\mu_{s}^{X},\mu_{s}^{X,N} )+C_{p,\theta}\mathcal{W}_{\theta}^{p}(\mathcal{L}_{X_{s}^{i}},\mu_{s}^{X} ).
\end{split}
\end{equation*}
By Lemma \ref{lem1}, we see
\begin{equation*}
\begin{split}
\boldsymbol{E}\left(\mathcal{W}_{\theta}^{p}(\mathcal{L}_{X_{s}^{i}},\mu_{s}^{X,N})\right)
\leq&C_{p,\theta}\boldsymbol{E}\left(\left(\frac{1}{N}\sum_{j=1}^{N}\left|X_{t}^{j}-X_{t}^{j,N}\right|^{\theta}\right)^{p/\theta}\right)
+C_{p,\theta}\boldsymbol{E}\left(\mathcal{W}_{\theta}^{p}(\mathcal{L}_{X_{s}^{i}},\mu_{s}^{X} )\right).
\end{split}
\end{equation*}
Thought the fact that $\mathcal{W}_\theta(\mu, \nu)\le \mathcal{W}_p(\mu, \nu)$ for $\theta\leq p$, we have
\begin{equation*}
\begin{split}
\boldsymbol{E}&\left(\left|X_{t}^{i}-X_{t}^{i,N}\right|^{p}\right)\\\leq&C_{p,T,L}\boldsymbol{E}\left(\int_{0}^{t}\left|X_{s}^{i}
-X_{s}^{i,N}\right|^{p}\mathrm{~d}s\right)
+C_{p,T,H,L,\theta}\int_{0}^{t}\boldsymbol{E}\left(\left(\frac{1}{N}\sum_{j=1}^{N}\left|X_{t}^{j}-X_{t}^{j,N}\right|^{\theta}\right)^{p/\theta}
\right)\mathrm{~d}s
\\&+C_{p,T,H,L,\theta}\int_{0}^{t}\boldsymbol{E}\left(\mathcal{W}_{p}^{p}(\mathcal{L}_{X_{s}^{i}},\mu_{s}^{X} )\right)
\mathrm{~d}s
\\\leq&C_{p,T,L}\boldsymbol{E}\left(\int_{0}^{t}\left|X_{s}^{i}-X_{s}^{i,N}\right|^{p}\mathrm{~d}s\right)
+C_{p,T,H,L,\theta}\int_{0}^{t}\left(\frac{1}{N}\sum_{j=1}^{N}\left\|\left|X_{s}^{j}
-X_{s}^{j,N}\right|^{\theta}\right\|_{L^{p/\theta}}\right)^{p/\theta}\mathrm{~d}s
\\&+C_{p,T,H,L,\theta}\int_{0}^{t}\boldsymbol{E}\left(\mathcal{W}_{p}^{p}(\mathcal{L}_{X_{s}^{i}},\mu_{s}^{X} )\right)
\mathrm{~d}s,
\end{split}
\end{equation*}
where we use the Minkowski inequality in the last inequality. Then through simple sorting, we have
\begin{equation*}
\begin{split}
\boldsymbol{E}&\left(\left|X_{t}^{i}-X_{t}^{i,N}\right|^{p}\right)\\\leq&C_{p,T,L}\boldsymbol{E}\left(\int_{0}^{t}\left|X_{s}^{i}-X_{s}^{i,N}\right|^{p}\mathrm{~d}s\right)
+C_{p,T,H,L,\theta}\int_{0}^{t}\boldsymbol{E}\left(\left|X_{s}^{i}
-X_{s}^{i,N}\right|^{p}\right)\mathrm{~d}s
\\&+C_{p,T,H,L,\theta}\int_{0}^{t}\boldsymbol{E}\left(\mathcal{W}_{p}^{p}(\mathcal{L}_{X_{s}^{i}},\mu_{s}^{X} )\right)
\mathrm{~d}s
\\\leq&C_{p,T,H,L,\theta}\int_{0}^{t}\boldsymbol{E}\left(\left|X_{s}^{i}
-X_{s}^{i,N}\right|^{p}\right)\mathrm{~d}s
+C_{p,T,H,L,\theta}\int_{0}^{t}\boldsymbol{E}\left(\mathcal{W}_{p}^{p}(\mathcal{L}_{X_{s}^{i}},\mu_{s}^{X} )\right)
\mathrm{~d}s.
\end{split}
\end{equation*}
What's particularly interesting is that $\mathcal{W}_{p}^{p}(\mathcal{L}_{X_{s}^{i}},\mu_{s}^{X} )$ is controlled by the Wasserstein distance estimate  in \cite[Theorem 1]{fournier2015}. Therefore,
\begin{equation*}
\begin{split}
\boldsymbol{E}\left(\mathcal{W}_{p}^{p}(\mathcal{L}_{X_{s}^{i}},\mu_{s}^{X} )\right)\leq& CM^{p/q}_{q}(\mathcal{L}_{X_{s}^{i}})\\&
\times\left\{\begin{array}{ll}
N^{-1 / 2}+N^{-(q-p)/q}, & \text { if } p>d/2\text{ and }q\neq2p, \\
N^{-1 / 2} \log(1+N)+N^{-(q-p)/q}, & \text { if } p=d/2\text{ and }q\neq2p,\\
N^{-p / d}+N^{-(q-p)/q}, & \text { if }p\in[2,d/2)\text{ and }q\neq d/(d-p),
\end{array}\right.
\end{split}
\end{equation*}
where $$M_{q}(\mathcal{L}_{X_{s}^{i}})=\int_{\mathbb{R}^{d}}\left|X_{s}^{i}\right|^{q}\mathcal{L}_{X_{s}^{i}}(\mathrm{~d}X_{s}^{i}).$$
By Theorem \ref{Xbound}, we note that $M_{q}(\mathcal{L}_{X_{s}^{i}})\leq C$.
Thus,
\begin{equation*}
\begin{split}
\boldsymbol{E}\left(\mathcal{W}_{p}^{p}(\mathcal{L}_{X_{s}^{i}},\mu_{s}^{X} )\right)\leq C
\left\{\begin{array}{ll}
N^{-1 / 2}+N^{-(q-p)/q}, & \text { if } p>d/2\text{ and }q\neq2p, \\
N^{-1 / 2} \log(1+N)+N^{-(q-p)/q}, & \text { if } p=d/2\text{ and }q\neq2p,\\
N^{-p / d}+N^{-(q-p)/q}, & \text { if }p\in[2,d/2)\text{ and }q\neq d/(d-p).
\end{array}\right.
\end{split}
\end{equation*}
Then, applying the Gr\"{o}nwall inequality completes the proof.\eproof
\end{proof}

\subsubsection{EM Method for Interacting Particle System}

Now, we define a uniform mesh $T_{N}: 0 = t_{0} < t_{1} < \cdots < t_{N} = T$ with $t_{k} = k\Delta$, where $\Delta= T /N$ for $N \in \mathbb{N}$. The numerical solutions are then generated by the EM method
\begin{equation}\label{McKeanNumerical}
Y_{t_{k+1}}^{i,N}=Y_{t_{k}}^{i,N}+b\left(Y_{t_{k}}^{i,N}, \mu_{t_{k}}^{Y,N}\right) \Delta+ \sigma\left( \mu_{t_{k}}^{Y,N}\right) \Delta B^{H,i}_{t_{k}},
\end{equation}
where the empirical measures
$
\mu_{t_{k}}^{Y,N}(\cdot):=\frac{1}{N} \sum_{j=1}^{N} \delta_{Y_{t_{k}}^{j,N}}(\cdot) \text{ and } \Delta B^{H,i}_{t_{k}}=B^{H,i}_{t_{k+1}}-B^{H,i}_{t_{k}}.
$
We show two versions of extension of the numerical solution at the discrete time points to $t\geq0$. The first is the piecewise constant extension given by
\begin{equation}
\bar{Y}_{t}^{i,N}= Y_{t_{k}}^{i,N}, \quad t_{k}\leq t < t_{k+1},
\end{equation}
and the second is the continuous extension of the EM method defined by
\begin{equation}\label{INn}
 Y_{t}^{i, N}=\bar{Y}_{t_{k}}^{i,N}+\int_{t_{k}}^{t} b\left(\bar{Y}_{s}^{i,N}, \bar{\mu}_{s}^{Y,N}\right) \mathrm{~d}s+\int_{t_{k}}^{t} \sigma\left(\bar{\mu}_{s}^{Y,N}\right)\mathrm{~d} B_{s}^{H,i}.
\end{equation}
Here,
$
\bar{\mu}_t^{Y,N}(\cdot):=\frac{1}{N}\sum_{j=1}^N \delta_{\bar{Y}_t^{j,N}}(\cdot).
$
From \eqref{INn}, for all $t\in[0,T]$, we have
\begin{equation}\label{INn0}
Y_{t}^{i, N}=X_{0}^{i}+\int_{0}^{t} b\left(\bar{Y}_{s}^{i,N}, \bar{\mu}_{s}^{Y,N}\right) \mathrm{~d}s+\int_{0}^{t} \sigma\left(\bar{\mu}_{s}^{Y,N}\right)\mathrm{~d} B_{s}^{H,i}.
\end{equation}

\begin{theorem}\label{ybound}
Let Assumption \ref{assp} holds. For some $q>p$, then
 \begin{equation*}
 \sup_{t\in[0,T]}\boldsymbol{E}\left(|Y^{i,N}_{t}|^{q}\right)\leq C_{q,H,T,L}\left(1+\boldsymbol{E}|X^{i}_{0}|^{q}\right),
\end{equation*}
where $C_{q,H,T,L}$ is a positive constant dependent on $q,H,T,L$ but independent of $\Delta$.
\end{theorem}
\begin{proof}
Similar to the proof of Theorem \ref{Xbound}, we can show
\begin{equation*}
\begin{split}
\boldsymbol{E}\left(|Y^{i,N}_{t}|^{q}\right)\leq& 3^{q-1}\boldsymbol{E}\left(|X^{i}_{0}|^{q}\right)+3^{q-1}\boldsymbol{E}\left(\left|\int_{0}^{t}b\left(\bar{Y}^{i,N}_{s},\bar{\mu}_{s}^{Y,N}\right)\mathrm{d}s\right|^{q}\right)
+3^{q-1}\boldsymbol{E}\left(\left|\int_{0}^{t}\sigma\left(\bar{\mu}_{s}^{Y,N}\right)\mathrm{d}B^{H,i}_{s}\right|^{q}\right)
\\\leq& C_{q,H,T,L}\left(1+\boldsymbol{E}|X^{i}_{0}|^{q}\right)+C_{q,H,T,L}\int_{0}^{t}\boldsymbol{E}\left(|\bar{Y}^{i,N}_{s}|^{q}\right)\mathrm{~d}s
\\\leq& C_{q,H,T,L}\left(1+\boldsymbol{E}|X^{i}_{0}|^{q}\right)
+C_{q,H,T,L}\int_{0}^{t}\sup_{0\leq\tau\leq s}\boldsymbol{E}\left(|Y^{i,N}_{s}|^{q}\right)\mathrm{~d}s
.
\end{split}
\end{equation*}
Therefore, for $0\leq t\leq T$, we have
\begin{equation*}
  \sup_{0\leq\tau\leq t}\boldsymbol{E}\left(|Y^{i,N}_{\tau}|^{q}\right)\leq C_{q,H,T,L}\left(1+\boldsymbol{E}|X^{i}_{0}|^{q}\right)
+C_{q,H,T,L}\int_{0}^{t}\sup_{0\leq\tau\leq s}\boldsymbol{E}\left(|Y^{i,N}_{s}|^{q}\right)\mathrm{~d}s
.
\end{equation*}
By the Gr\"{o}nwall inequality, we see
\begin{equation*}
  \sup_{0\leq\tau\leq t}\boldsymbol{E}\left(|Y^{i,N}_{\tau}|^{q}\right)\leq C_{q,H,T,L}\left(1+\boldsymbol{E}|X^{i}_{0}|^{q}\right), \quad \forall t\in[0,T].
\end{equation*}
Therefore, the assertion holds.\eproof
\end{proof}
\begin{lemma}\label{XY}
Assume \eqref{tui1} and \eqref{tui2} hold, for a constant $\kappa\in(1-H,1-1/p)$, then
\begin{equation}\label{xyy}
\boldsymbol{E}\left(\left|Y^{i,N}_{t}-\bar{Y}^{i,N}_{t}\right|^{p}\right)\leq C_{\kappa,H,p,L}\Delta^{pH},
\end{equation}
where $C_{\kappa,H,p,L}$is a positive constant dependent on $\kappa,H,p,L$ but independent of $\Delta$.
\end{lemma}
\begin{proof}
From \eqref{INn}, we separate the left hand side of \eqref{xyy} into two parts
\begin{equation}\label{PP}
\begin{split}
 \boldsymbol{E}\left(\left|Y^{i,N}_{t}-\bar{Y}^{i,N}_{t}\right|^{p}\right)=& \boldsymbol{E}\left(\left|\int_{t_{k}}^{t} b\left(\bar{Y}_{s}^{i,N}, \bar{\mu}_{s}^{Y,N}\right) \mathrm{~d}s+\int_{t_{k}}^{t} \sigma\left(\bar{\mu}_{s}^{Y,N}\right)\mathrm{~d} B_{s}^{H,i}\right|^{p}\right)
 \\\leq&2^{p-1}\boldsymbol{E}\left(\left|\int_{t_{k}}^{t} b\left(\bar{Y}_{s}^{i,N}, \bar{\mu}_{s}^{Y,N}\right) \mathrm{~d}s\right|^{p}\right)+2^{p-1}\boldsymbol{E}\left(\left|\int_{t_{k}}^{t} \sigma\left(\bar{\mu}_{s}^{Y,N}\right)\mathrm{~d} B_{s}^{H,i}\right|^{p}\right).
\end{split}
\end{equation}
Let us first consider the first part on the right of \eqref{PP}, by using the H\"{o}lder inequality and \eqref{tui1}, we obtain
\begin{equation}\label{P1}
\begin{split}
\boldsymbol{E}&\left(\left|\int_{t_{k}}^{t} b\left(\bar{Y}_{s}^{i,N}, \bar{\mu}_{s}^{Y,N}\right) \mathrm{~d}s\right|^{p}\right)
\\\leq&\Delta^{p-1}\boldsymbol{E}\left(\int_{t_{k}}^{t} \left|b\left(\bar{Y}_{s}^{i,N}, \bar{\mu}_{s}^{Y,N}\right)\right|^{p} \mathrm{~d}s\right)
\\\leq&\Delta^{p-1}\boldsymbol{E}\left(\int_{t_{k}}^{t}L^{p} \left(1+\left|\bar{Y}_{s}^{i,N}\right|+\mathcal{W}_{\theta}\left(\bar{\mu}_{s}^{Y,N},\delta_{0}\right)\right)^{p} \mathrm{~d}s\right)
\\\leq&L^{p}\Delta^{p}+C_{p,L}\Delta^{p-1}\int_{t_{k}}^{t}\boldsymbol{E}\left(\left|\bar{Y}_{s}^{i,N}\right|^{p}\right)\mathrm{~d}s
+C_{p,L}\Delta^{p-1}\int_{t_{k}}^{t}\boldsymbol{E}\left(\mathcal{W}_{\theta}^{p}\left(\bar{\mu}_{s}^{Y,N},\delta_{0}\right)\right)\mathrm{~d}s.
\end{split}
\end{equation}
For  the second part on the right of \eqref{PP}, by choosing $\kappa\in(1-H,1-1/p)$, we note the fact that $\int_{s}^{t}(t-r)^{-\kappa}(r-s)^{\kappa-1}dr=C_{\kappa}$. Then
\begin{equation*}
\begin{split}
\boldsymbol{E}\left(\left|\int_{t_{k}}^{t} \sigma\left(\bar{\mu}_{s}^{Y,N}\right)\mathrm{~d} B_{s}^{H,i}\right|^{p}\right)
=C_{\kappa}^{-p}\boldsymbol{E}\left(\left|\int_{t_{k}}^{t}\left(\int_{s}^{t}(t-r)^{-\kappa}(r-s)^{\kappa-1}\mathrm{~d}r\right) \sigma\left(\bar{\mu}_{s}^{Y,N}\right)\mathrm{~d} B_{s}^{H,i}\right|^{p}\right).
\end{split}
\end{equation*}
Thanks to Stochastic Fubini Theorem for the Wiener Integrals with regard to fBm \cite[Theorem 1.13.1]{mishura2008}, we obtain
\begin{equation*}
\int_{t_{k}}^{t}\left(\int_{s}^{t}(t-r)^{-\kappa}(r-s)^{\kappa-1}\mathrm{~d}r\right) \sigma\left(\bar{\mu}_{s}^{Y,N}\right)\mathrm{~d} B_{s}^{H,i}=\int_{t_{k}}^{t}(t-r)^{-\kappa}\left(\int_{0}^{r}(r-s)^{\kappa-1} \sigma\left(\bar{\mu}_{s}^{Y,N}\right)\mathrm{~d} B_{s}^{H,i}\right)\mathrm{~d}r.
\end{equation*}
Therefore, by the H\"{o}lder inequality, we get
\begin{equation*}
\begin{split}
\boldsymbol{E}&\left(\left|\int_{t_{k}}^{t} \sigma\left(\bar{\mu}_{s}^{Y,N}\right)\mathrm{~d} B_{s}^{H,i}\right|^{p}\right)\\\leq&C_{\kappa}^{-p}\boldsymbol{E}
\left(\left(\int_{t_{k}}^{t}(t-r)^{-p\kappa/(p-1)}\mathrm{~d}r\right)^{p-1}\int_{t_{k}}^{t}\left|\int_{0}^{r}(r-s)^{\kappa-1} \sigma\left(\bar{\mu}_{s}^{Y,N}\right)\mathrm{~d} B_{s}^{H,i}\right|^{p}\mathrm{~d}r\right)
\\\leq&\frac{C_{\kappa}^{-p}(p-1)}{(p-1-p\kappa)^{p-1}}\Delta^{p-1-p\kappa}\boldsymbol{E}\left(\int_{t_{k}}^{t}\left|\int_{0}^{r}(r-s)^{\kappa-1} \sigma\left(\bar{\mu}_{s}^{Y,N}\right)\mathrm{~d} B_{s}^{H,i}\right|^{p}\mathrm{~d}r\right) .
\end{split}
\end{equation*}
Applying the Theorem 1.1 in \cite{memin2001}, we see
\begin{equation*}
\boldsymbol{E}\left(\left|\int_{0}^{r}(r-s)^{\kappa-1} \sigma\left(\bar{\mu}_{s}^{Y,N}\right)\mathrm{~d} B_{s}^{H,i}\right|^{p}\right)
\leq C_{H,p}\left(\int_{0}^{r}(r-s)^{(\kappa-1)/H}\left|\sigma\left(\bar{\mu}_{s}^{Y,N}\right)\right|^{1/H}\mathrm{~d}s\right)^{pH}.
\end{equation*}
Thus, we obtain
\begin{equation*}
\boldsymbol{E}\left(\left|\int_{t_{k}}^{t} \sigma\left(\bar{\mu}_{s}^{Y,N}\right)\mathrm{~d} B_{s}^{H,i}\right|^{p}\right)\leq C_{\kappa,H,p}\Delta^{p-1-p\kappa}\int_{t_{k}}^{t}\left(\int_{0}^{r}(r-s)^{(\kappa-1)/H}
\left|\sigma\left(\bar{\mu}_{s}^{Y,N}\right)\right|^{1/H}\mathrm{~d}s\right)^{pH}\mathrm{~d}r.
\end{equation*}
By \cite[Lemma 3.2]{fan2022} with $\widetilde{q}=pH$, $\alpha=\frac{H-1-\kappa}{H}$ and $\widetilde{p}=\frac{pH}{p(\kappa+H-1)+1}$, we have
\begin{equation*}
\begin{split}
\boldsymbol{E}\left(\left|\int_{t_{k}}^{t} \sigma\left(\bar{\mu}_{s}^{Y,N}\right)\mathrm{~d} B_{s}^{H,i}\right|^{p}\right)\leq& C_{\kappa,H,p}\Delta^{p-1-p\kappa}
\left(\int_{t_{k}}^{t}\left|\sigma\left(\bar{\mu}_{s}^{Y,N}\right)\right|^{p/[p(\kappa+H-1)+1]}\mathrm{~d}s\right)^{p(\kappa+H-1)+1}
\\\leq& C_{\kappa,H,p}\Delta^{p-1-p\kappa}\Delta^{p(\kappa+H-1)}\int_{t_{k}}^{t}\left|\sigma\left(\bar{\mu}_{s}^{Y,N}\right)\right|^{p}\mathrm{~d}s,
\end{split}
\end{equation*}
where we use the H\"{o}lder inequality in the last inequality. Then by \eqref{tui2}, we know that
\begin{equation}\label{P2}
\begin{split}
\boldsymbol{E}\left(\left|\int_{t_{k}}^{t} \sigma\left(\bar{\mu}_{s}^{Y,N}\right)\mathrm{~d} B_{s}^{H,i}\right|^{p}\right)\leq& C_{\kappa,H,p}\Delta^{pH-1}\int_{t_{k}}^{t}\boldsymbol{E}\left(\left|\sigma\left(\bar{\mu}_{s}^{Y,N}\right)\right|^{p}\right)\mathrm{~d}s
\\\leq & C_{\kappa,H,p}\Delta^{pH-1}\int_{t_{k}}^{t}L^{p}\boldsymbol{E}\left(1+\mathcal{W}_{\theta}\left(\bar{\mu}_{s}^{Y,N},\delta_{0}\right)\right)^{p}\mathrm{~d}s
\\\leq & C_{\kappa,H,p,L}\Delta^{pH}
+C_{\kappa,H,p,L}\Delta^{pH-1}\int_{t_{k}}^{t}\boldsymbol{E}
\left(\mathcal{W}_{\theta}^{p}\left(\bar{\mu}_{s}^{Y,N},\delta_{0}\right)\right)\mathrm{~d}s.
\end{split}
\end{equation}
Combining \eqref{P1} and \eqref{P2} into \eqref{PP} yields
\begin{equation}\label{P1P2}
\begin{split}
 \boldsymbol{E}\left(\left|Y^{i,N}_{t}-\bar{Y}^{i,N}_{t}\right|^{p}\right)\leq&L^{p}\Delta^{p}
 +C_{p,L}\Delta^{p-1}\int_{t_{k}}^{t}\boldsymbol{E}\left(\left|\bar{Y}_{s}^{i,N}\right|^{p}\right)\mathrm{~d}s
+C_{p,L}\Delta^{p-1}\int_{t_{k}}^{t}\boldsymbol{E}\left(\mathcal{W}_{\theta}^{p}\left(\bar{\mu}_{s}^{Y,N},\delta_{0}\right)\right)\mathrm{~d}s
\\+&C_{\kappa,H,p,L}\Delta^{pH}
+C_{\kappa,H,p,L}\Delta^{pH-1}\int_{t_{k}}^{t}\boldsymbol{E}
\left(\mathcal{W}_{\theta}^{p}\left(\bar{\mu}_{s}^{Y,N},\delta_{0}\right)\right)\mathrm{~d}s.
\end{split}
\end{equation}
By Lemma \ref{lem1} and the Minkowski inequality , we note that
\begin{equation}\label{Wp}
\begin{split}
\boldsymbol{E}\left(\mathcal{W}_{\theta}^{p}\left(\bar{\mu}_{s}^{Y,N},\delta_{0}\right)\right)
=\boldsymbol{E}\left(\left(\mathcal{W}_{\theta}^{\theta}\left(\bar{\mu}_{s}^{Y,N},\delta_{0}\right)\right)^{p/\theta}\right)
\leq\boldsymbol{E}\left(\left(\frac{1}{N}\sum_{j=1}^{N}\left|\bar{Y}_{s}^{j,N}\right|^{\theta}\right)^{p/\theta}\right)
=\boldsymbol{E}\left(\left|\bar{Y}_{s}^{i,N}\right|^{p}\right).
\end{split}
\end{equation}
Substituting \eqref{Wp} into \eqref{P1P2}, we get
\begin{equation*}
\begin{split}
 \boldsymbol{E}\left(\left|Y^{i,N}_{t}-\bar{Y}^{i,N}_{t}\right|^{p}\right)\leq&L^{p}\Delta^{p}
 +C_{p,L}\Delta^{p-1}\int_{t_{k}}^{t}\boldsymbol{E}\left(\left|\bar{Y}_{s}^{i,N}\right|^{p}\right)\mathrm{~d}s
+C_{p,L}\Delta^{p-1}\int_{t_{k}}^{t}\boldsymbol{E}\left(\left|\bar{Y}_{s}^{i,N}\right|^{p}\right)\mathrm{~d}s
\\&+C_{\kappa,H,p,L}\Delta^{pH}
+C_{\kappa,H,p,L}\Delta^{pH-1}\int_{t_{k}}^{t}\boldsymbol{E}\left(\left|\bar{Y}_{s}^{i,N}\right|^{p}\right)\mathrm{~d}s
\\\leq&C_{\kappa,H,p,L}\Delta^{pH}
+C_{\kappa,H,p,L}\Delta^{pH-1}\int_{t_{k}}^{t}\boldsymbol{E}\left(\left|\bar{Y}_{s}^{i,N}\right|^{p}\right)\mathrm{~d}s.
\end{split}
\end{equation*}
Applying the Theorem \ref{ybound} to above inequality gives that
\begin{equation*}
\boldsymbol{E}\left(\left|Y^{i,N}_{t}-\bar{Y}^{i,N}_{t}\right|^{p}\right)\leq C_{\kappa,H,p,L}\Delta^{pH}.
\end{equation*}
Thus, we obtain the desired result.\eproof
\end{proof}

\begin{theorem}\label{thm}
Let Assumption \ref{assp} holds, for $p\geq\theta$, then
\begin{equation*}
\boldsymbol{E}\left(\left|X^{i,N}_{t}-Y^{i,N}_{t}\right|^{p}\right)\leq C_{p,T,H,L,\kappa}\Delta^{pH},
\end{equation*}
where $C_{p,T,H,L,\kappa}$ is a positive constant independent of $\Delta$.
\end{theorem}
\begin{proof}
 It follows from \eqref{McKeanI} and \eqref{INn0} that
\begin{equation*}
\begin{split}
\boldsymbol{E}&\left(\left|X^{i,N}_{t}-Y^{i,N}_{t}\right|^{p}\right)\\=&\boldsymbol{E}\left(\left|\int_{0}^{t}\left(b\left(X_{s}^{i,N}, \mu_{s}^{X,N}\right) -b\left(\bar{Y}_{s}^{i,N}, \bar{\mu}_{s}^{Y,N}\right) \right)\mathrm{~d}s+\int_{0}^{t} \left(\sigma\left(\mu_{s}^{X,N}\right) -\sigma\left(\bar{\mu}_{s}^{Y,N}\right)\right)\mathrm{~d} B_{s}^{H,i}\right|^{p}\right)
\\\leq&2^{p-1}\boldsymbol{E}\left(\left|\int_{0}^{t}\left(b\left(X_{s}^{i,N}, \mu_{s}^{X,N}\right) -b\left(\bar{Y}_{s}^{i,N}, \bar{\mu}_{s}^{Y,N}\right) \right)\mathrm{~d}s\right|^{p}\right)
\\&+2^{p-1}\boldsymbol{E}\left(\left|\int_{0}^{t} \left(\sigma\left(\mu_{s}^{X,N}\right) -\sigma\left(\bar{\mu}_{s}^{Y,N}\right)\right)\mathrm{~d} B_{s}^{H,i}\right|^{p}\right).
\end{split}
\end{equation*}
Thanks to the H\"{o}lder inequality and Assumption \ref{assp} yield
\begin{equation}\label{HH}
\begin{split}
\boldsymbol{E}&\left(\left|\int_{0}^{t}\left(b\left(X_{s}^{i,N}, \mu_{s}^{X,N}\right) -b\left(\bar{Y}_{s}^{i,N}, \bar{\mu}_{s}^{Y,N}\right) \right)\mathrm{~d}s\right|^{p}\right)\\\leq&t^{p-1}\boldsymbol{E}\left(\int_{0}^{t}\left|\left(b\left(X_{s}^{i,N}, \mu_{s}^{X,N}\right) -b\left(\bar{Y}_{s}^{i,N}, \bar{\mu}_{s}^{Y,N}\right) \right)\right|^{p}\mathrm{~d}s\right)
\\\leq&T^{p-1}\boldsymbol{E}\left(\int_{0}^{t}L^{p}\left(\left|X_{s}^{i,N}-\bar{Y}_{s}^{i,N}\right|
+\mathcal{W}_{\theta}\left(\mu_{s}^{X,N},\bar{\mu}_{s}^{Y,N}\right)\right)^{p}\mathrm{~d}s\right)
\\\leq&C_{p,T,L}\boldsymbol{E}\left(\int_{0}^{t}\left|X_{s}^{i,N}-\bar{Y}_{s}^{i,N}\right|^{p}\mathrm{~d}s\right)
+C_{p,T,L}\boldsymbol{E}\left(\int_{0}^{t}\mathcal{W}_{\theta}^{p}\left(\mu_{s}^{X,N},\bar{\mu}_{s}^{Y,N}\right)\mathrm{~d}s\right).
\end{split}
\end{equation}
Following a very similar approach used for \eqref{BH}, we can show
\begin{equation*}
\boldsymbol{E}\left(\left|\int_{0}^{t} \left(\sigma\left(\mu_{s}^{X,N}\right) -\sigma\left(\bar{\mu}_{s}^{Y,N}\right)\right)\mathrm{~d} B_{s}^{H,i}\right|^{p}\right)\leq  C_{p,T,H,L}\int_{0}^{t}\boldsymbol{E}\left(\mathcal{W}_{\theta}^{p}(\mu_{s}^{X,N} ,\bar{\mu}_{s}^{Y,N})\right)\mathrm{~d}s.
\end{equation*}
Thus,
\begin{equation}\label{PPP}
\boldsymbol{E}\left(\left|X^{i,N}_{t}-Y^{i,N}_{t}\right|^{p}\right)
\leq C_{p,T,L}\boldsymbol{E}\left(\int_{0}^{t}\left|X_{s}^{i,N}-\bar{Y}_{s}^{i,N}\right|^{p}\mathrm{~d}s\right)
+C_{p,T,H,L}\int_{0}^{t}\boldsymbol{E}\left(\mathcal{W}_{\theta}^{p}(\mu_{s}^{X,N} ,\bar{\mu}_{s}^{Y,N})\right)\mathrm{~d}s.
\end{equation}
Due to Lemma \ref{lem1} and the Minkowski inequality, we observe 
\begin{equation}\label{Wpp}
\begin{split}
\boldsymbol{E}\left(\mathcal{W}_{\theta}^{p}(\mu_{s}^{X,N} ,\bar{\mu}_{s}^{Y,N})\right)
=&\boldsymbol{E}\left(\left(\mathcal{W}_{\theta}^{\theta}(\mu_{s}^{X,N} ,\bar{\mu}_{s}^{Y,N})\right)^{p/\theta}\right)
\\\leq&\boldsymbol{E}\left(\left(\frac{1}{N}\sum_{j=1}^{N}\left|X_{s}^{j,N}-\bar{Y}_{s}^{j,N}\right|^{\theta}\right)^{p/\theta}\right)
\\\leq&\left(\frac{1}{N}\sum_{j=1}^{N}\left\|\left|X_{s}^{j,N}-\bar{Y}_{s}^{j,N}\right|^{\theta}\right\|_{L^{p/\theta}}\right)^{p/\theta}
\\=&\boldsymbol{E}\left(\left|X_{s}^{i,N}-\bar{Y}_{s}^{i,N}\right|^{p}\right).
\end{split}
\end{equation}
Substituting \eqref{Wpp} into \eqref{PPP}, we get
\begin{equation*}
\begin{split}
\boldsymbol{E}\left(\left|X^{i,N}_{t}-Y^{i,N}_{t}\right|^{p}\right)
\leq&C_{p,T,L}\int_{0}^{t}\boldsymbol{E}\left(\left|X_{s}^{i,N}-\bar{Y}_{s}^{i,N}\right|^{p}\right)\mathrm{~d}s
+C_{p,T,H,L}\int_{0}^{t}\boldsymbol{E}\left(\left|X_{s}^{i,N}-\bar{Y}_{s}^{i,N}\right|^{p}\right)\mathrm{~d}s
\\\leq&C_{p,T,H,L}\int_{0}^{t}\boldsymbol{E}\left(\left|X_{s}^{i,N}-\bar{Y}_{s}^{i,N}\right|^{p}\right)\mathrm{~d}s
\\\leq&C_{p,T,H,L}\int_{0}^{t}\boldsymbol{E}\left(\left|X_{s}^{i,N}-Y_{s}^{i,N}\right|^{p}\right)\mathrm{~d}s
+C_{p,T,H,L}\int_{0}^{t}\boldsymbol{E}\left(\left|Y_{s}^{i,N}-\bar{Y}_{s}^{i,N}\right|^{p}\right)\mathrm{~d}s.
\end{split}
\end{equation*}
We derive from Lemma \ref{XY} that
\begin{equation*}
\boldsymbol{E}\left(\left|X^{i,N}_{t}-Y^{i,N}_{t}\right|^{p}\right)
\leq C_{p,T,H,L}\int_{0}^{t}\boldsymbol{E}\left(\left|X_{s}^{i,N}-Y_{s}^{i,N}\right|^{p}\right)\mathrm{~d}s
+C_{p,T,H,L,\kappa}\Delta^{pH}.
\end{equation*}
By the Gr\"{o}nwall inequality, we derive that
\begin{equation*}
\boldsymbol{E}\left(\left|X^{i,N}_{t}-Y^{i,N}_{t}\right|^{p}\right)
\leq C_{p,T,H,L,\kappa}\Delta^{pH}.
\end{equation*}
Therefore, the proof is complete.\eproof
\end{proof}

\begin{theorem}
Let Assumption \ref{assp} be satisfied. If $\theta\leq p< q$, then it holds that
\begin{equation*}
\begin{split}
\sup _{i \in\{1, \ldots, N\}} \sup _{t \in[0, T]} \boldsymbol{E}&\left(\left|X_{t}^{i}-Y_{t}^{i,N}\right|^{p}\right) \leq C_{p,T,H,L,\theta,\kappa}\\&
\times
\left\{\begin{array}{ll}
N^{-1 / 2}+N^{-(q-p)/q}+\Delta^{pH}, & \text { if } p>d/2\text{ and }q\neq2p, \\
N^{-1 / 2} \log(1+N)+N^{-(q-p)/q}+\Delta^{pH}, & \text { if } p=d/2\text{ and }q\neq2p,\\
N^{-p / d}+N^{-(q-p)/q}+\Delta^{pH}, & \text { if }p\in[2,d/2)\text{ and }q\neq d/(d-p),
\end{array}\right.
\end{split}
\end{equation*}
where the constant $C_{p,T,H,L,\theta,\kappa}> 0$ does not depend on $N$ and $\Delta$.
\end{theorem}
We can easily get this theorem through trigonometric inequality, Theorem \ref{pc} and Theorem \ref{thm}.

\subsection{Case $H<1/2$}
For the case of $H\in(0,1/2)$, due to the Theorem 3.1(II) in \cite{fan2022}, $\sigma(\mu)$ is independent of distribution and then the solution of $X_{t}$ in \eqref{McKean} exists and is unique. In other words, the coefficient of diffusion $\sigma\left(\mu_{s}^{X,N}\right)=\xi$, where $\xi$ is a constant.  Here, we only consider the case of $p= 2$ (imply $\theta=2$).

\begin{theorem}\label{Xbound1}
Let Assumption \ref{assp} holds, for $q>2$ then
\begin{equation*}
  \sup_{t\in[0,T]}\E\left(|X^{i}_{t}|^{q}\right)+\sup_{t\in[0,T]}\boldsymbol{E}\left(|X^{i,N}_{t}|^{q}\right)\leq C_{q,T,L,H,\xi}\left(1+\boldsymbol{E}|X^{i}_{0}|^{q}\right),
\end{equation*}
where $C_{q,T,L,H,\xi}$ is a positive constant dependent on $q,T,L,H,\xi$ but does not depend on $N$.
\end{theorem}
\begin{proof}
From \eqref{McKeanI}, \eqref{J1} and Theorem 2.1 in \cite{yaskov2019}, we get
\begin{equation*}
\begin{split}
\boldsymbol{E}\left(|X^{i,N}_{t}|^{q}\right)\leq& 3^{q-1}\boldsymbol{E}\left(|X^{i}_{0}|^{q}\right)+3^{q-1}\boldsymbol{E}\left(\left|\int_{0}^{t}b\left(X^{i,N}_{s},\mu_{s}^{X,N}\right)\mathrm{d}s\right|^{2}\right)
+3^{q-1}\boldsymbol{E}\left(\left|\int_{0}^{t}\xi\mathrm{d}B^{H,i}_{s}\right|^{q}\right)
\\\leq&C_{q}\boldsymbol{E}\left(|X^{i}_{0}|^{2}\right)+C_{q,T,L}\int_{0}^{t}\boldsymbol{E}\left(|X^{i,N}_{s}|^{2}\right)\mathrm{~d}s
+C_{q,T,L}\int_{0}^{t}\boldsymbol{E}\left(\mathcal{W}_{2}^{q}(\mu_{s}^{X,N},\delta_{0} )\right)\mathrm{d}s+C_{q}\xi^{q} T^{qH}
\\\leq& C_{q,T,L,H,\xi}\left(1+\boldsymbol{E}|X^{i}_{0}|^{q}\right)+C_{T,L}\int_{0}^{t}\boldsymbol{E}\left(|X^{i,N}_{s}|^{q}\right)\mathrm{~d}s.
\end{split}
\end{equation*}
 Then applying the Gr\"{o}nwall inequality, we obtain
 \begin{equation*}
\boldsymbol{E}\left(|X^{i,N}_{t}|^{q}\right)\leq C_{q,T,L,H,\xi}\left(1+\boldsymbol{E}|X^{i}_{0}|^{q}\right).
\end{equation*}
 Similarly, we can show
 \begin{equation*}
\boldsymbol{E}\left(|X^{i}_{t}|^{q}\right)\leq C_{q,T,L,H,\xi}\left(1+\boldsymbol{E}|X^{i}_{0}|^{q}\right).
\end{equation*}
The proof is complete.\eproof
\end{proof}

\begin{theorem}
[Propagation of Chaos]\label{pc1}
Let Assumption \ref{assp} be satisfied. If for some $q>2$, then it holds that
\begin{equation*}
\begin{split}
\sup _{i \in\{1, \ldots, N]} \sup _{t \in[0, T]} \boldsymbol{E}\left|X_{t}^{i}-X_{t}^{i,N}\right|^{2} &\leq C_{T,H,L}\\&
\times
\left\{\begin{array}{ll}
N^{-1 / 2}+N^{-(q-2)/q}, & \text { if } d<4,\\
N^{-1 / 2} \log(1+N)+N^{-(q-2)/q}, & \text { if } d=4,\\
N^{-2 / d}+N^{-(q-2)/q}, & \text { if }d>4,
\end{array}\right.
\end{split}
\end{equation*}
where the constant $C_{T,H,L}> 0$ depends on $T$, $H$ and $L$ but does not depend on $N$.
\end{theorem}
The proof of this lemma is similar to that of Theorem \ref{pc}, we put it in the Appendix \ref{pcc1}.

\subsubsection{EM Method for Interacting Particle System}

The numerical solutions are generated by the EM method
\begin{equation}\label{McKeanNumerical1}
Y_{t_{k+1}}^{i,N}=Y_{t_{k}}^{i,N}+b\left(Y_{t_{k}}^{i,N}, \mu_{t_{k}}^{Y,N}\right) \Delta+ \xi \Delta B^{H,i}_{t_{k}},
\end{equation}
and the continuous extension of the EM method defined by
\begin{equation}\label{INn1}
 Y_{t}^{i, N}=\bar{Y}_{t_{k}}^{i,N}+\int_{t_{k}}^{t} b\left(\bar{Y}_{s}^{i,N}, \bar{\mu}_{s}^{Y,N}\right) \mathrm{~d}s+\int_{t_{k}}^{t} \xi\mathrm{~d} B_{s}^{H,i}.
\end{equation}
From \eqref{INn1}, for all $t\in[0,T]$, we have
\begin{equation}\label{INn01}
Y_{t}^{i, N}=X_{0}^{i}+\int_{0}^{t} b\left(\bar{Y}_{s}^{i,N}, \bar{\mu}_{s}^{Y,N}\right) \mathrm{~d}s+\int_{0}^{t} \xi\mathrm{~d} B_{s}^{H,i}.
\end{equation}
\begin{theorem}\label{ybound1}
Let Assumption \ref{assp} holds,  then
 \begin{equation*}
 \sup_{t\in[0,T]}\boldsymbol{E}\left(|Y^{i,N}_{t}|^{2}\right)\leq C_{T,L,H,\xi}\left(1+\boldsymbol{E}|X^{i}_{0}|^{2}\right),
\end{equation*}
where $C_{T,L,H,\xi}$ is a positive constant dependent on $T,L,H,\xi$ but independent of $\Delta$.
\end{theorem}
\begin{proof}
By \eqref{J1} and \eqref{INn01}, we can show
\begin{equation*}
\begin{split}
\boldsymbol{E}\left(|Y^{i,N}_{t}|^{2}\right)\leq& 3\boldsymbol{E}\left(|X^{i}_{0}|^{2}\right)+3\boldsymbol{E}\left(\left|\int_{0}^{t}b\left(\bar{Y}^{i,N}_{s},\bar{\mu}_{s}^{Y,N}\right)\mathrm{d}s\right|^{2}\right)
+3\boldsymbol{E}\left(\left|\int_{0}^{t}\xi\mathrm{d}B^{H,i}_{s}\right|^{2}\right)
\\\leq&3\boldsymbol{E}\left(|X^{i}_{0}|^{2}\right)+C_{T,L}\int_{0}^{t}\boldsymbol{E}\left(|\bar{Y}^{i,N}_{s}|^{2}\right)\mathrm{~d}s
+C_{T,L}\int_{0}^{t}\boldsymbol{E}\left(\mathcal{W}_{2}^{2}(\bar{\mu}_{s}^{Y,N},\delta_{0} )\right)\mathrm{d}s+3\xi^{2} t^{2H}
\\\leq& C_{T,L,H,\xi}\left(1+\boldsymbol{E}|X^{i}_{0}|^{2}\right)+C_{T,L}\int_{0}^{t}\boldsymbol{E}\left(|\bar{Y}^{i,N}_{s}|^{2}\right)\mathrm{~d}s
\\\leq& C_{T,L,H,\xi}\left(1+\boldsymbol{E}|X^{i}_{0}|^{2}\right)
+C_{T,L}\int_{0}^{t}\sup_{0\leq\tau\leq s}\boldsymbol{E}\left(|Y^{i,N}_{s}|^{2}\right)\mathrm{~d}s
.
\end{split}
\end{equation*}
Therefore, for $0\leq t\leq T$, we have
\begin{equation*}
  \sup_{0\leq\tau\leq t}\boldsymbol{E}\left(|Y^{i,N}_{\tau}|^{2}\right)\leq C_{T,L,H,\xi}\left(1+\boldsymbol{E}|X^{i}_{0}|^{2}\right)
+C_{T,L}\int_{0}^{t}\sup_{0\leq\tau\leq s}\boldsymbol{E}\left(|Y^{i,N}_{s}|^{2}\right)\mathrm{~d}s
.
\end{equation*}
By the Gr\"{o}nwall inequality, the assertion holds.\eproof
\end{proof}

\begin{lemma}\label{XY1}
Assume \eqref{tui1} and \eqref{tui2} hold, then
\begin{equation}\label{xyy1}
\boldsymbol{E}\left(\left|Y^{i,N}_{t}-\bar{Y}^{i,N}_{t}\right|^{2}\right)\leq C_{L}\Delta^{2H},
\end{equation}
where $C_{L}$ is a positive constant dependent on $L$ but independent of $\Delta$.
\end{lemma}
\begin{proof}
From \eqref{INn1}, we get
\begin{equation}\label{PP1}
\begin{split}
 \boldsymbol{E}\left(\left|Y^{i,N}_{t}-\bar{Y}^{i,N}_{t}\right|^{2}\right)=& \boldsymbol{E}\left(\left|\int_{t_{k}}^{t} b\left(\bar{Y}_{s}^{i,N}, \bar{\mu}_{s}^{Y,N}\right) \mathrm{~d}s+\int_{t_{k}}^{t} \xi\mathrm{~d} B_{s}^{H,i}\right|^{2}\right)
 \\\leq&2\boldsymbol{E}\left(\left|\int_{t_{k}}^{t} b\left(\bar{Y}_{s}^{i,N}, \bar{\mu}_{s}^{Y,N}\right) \mathrm{~d}s\right|^{2}\right)+2\boldsymbol{E}\left(\left|\int_{t_{k}}^{t} \xi\mathrm{~d} B_{s}^{H,i}\right|^{2}\right).
\end{split}
\end{equation}
By \eqref{P1} and the fact that $\boldsymbol{E}\left(\left|B^{H}_{t}-B^{H}_{s}\right|^{2}\right)=|t-s|^{2 H}$, we obtain
\begin{equation*}
\begin{split}
 \boldsymbol{E}\left(\left|Y^{i,N}_{t}-\bar{Y}^{i,N}_{t}\right|^{2}\right)\leq&C_{L}\Delta^{2}
 +C_{L}\Delta\int_{t_{k}}^{t}\boldsymbol{E}\left(\left|\bar{Y}_{s}^{i,N}\right|^{2}\right)\mathrm{~d}s
+C_{L}\Delta\int_{t_{k}}^{t}\boldsymbol{E}\left(\mathcal{W}_{2}^{2}\left(\bar{\mu}_{s}^{Y,N},\delta_{0}\right)\right)\mathrm{~d}s
\\&+C_{L}\Delta^{2H}
+C_{L}\Delta^{2H-1}\int_{t_{k}}^{t}\boldsymbol{E}
\left(\mathcal{W}_{2}^{2}\left(\bar{\mu}_{s}^{Y,N},\delta_{0}\right)\right)\mathrm{~d}s.
\\\leq&C_{L}\Delta^{2}
 +C_{L}\Delta\int_{t_{k}}^{t}\boldsymbol{E}\left(\left|\bar{Y}_{s}^{i,N}\right|^{2}\right)\mathrm{~d}s
+C_{L}\Delta\int_{t_{k}}^{t}\boldsymbol{E}\left(\left|\bar{Y}_{s}^{i,N}\right|^{2}\right)\mathrm{~d}s
\\&+C_{L}\Delta^{2H}
+C_{L}\Delta^{2H-1}\int_{t_{k}}^{t}\boldsymbol{E}\left(\left|\bar{Y}_{s}^{i,N}\right|^{2}\right)\mathrm{~d}s
\\\leq&C_{L}\Delta^{2H}
+C_{L}\Delta^{2H-1}\int_{t_{k}}^{t}\boldsymbol{E}\left(\left|\bar{Y}_{s}^{i,N}\right|^{2}\right)\mathrm{~d}s.
\end{split}
\end{equation*}
Applying Theorem \ref{ybound1} to above inequality and then getting the desired result.\eproof
\end{proof}

\begin{theorem}\label{thm1}
Let Assumption \ref{assp} holds, then
\begin{equation*}
\boldsymbol{E}\left(\left|X^{i,N}_{t}-Y^{i,N}_{t}\right|^{2}\right)\leq C_{T,L}\Delta^{2H},
\end{equation*}
where $C_{T,L}$ is a positive constant independent of $\Delta$.
\end{theorem}
\begin{proof}
 It follows from \eqref{McKeanI} and \eqref{INn01} that
\begin{equation*}
\boldsymbol{E}\left(\left|X^{i,N}_{t}-Y^{i,N}_{t}\right|^{2}\right)=\boldsymbol{E}\left(\left|\int_{0}^{t}\left(b\left(X_{s}^{i,N}, \mu_{s}^{X,N}\right) -b\left(\bar{Y}_{s}^{i,N}, \bar{\mu}_{s}^{Y,N}\right) \right)\mathrm{~d}s\right|^{p}\right).
\end{equation*}
By \eqref{HH}, we have
\begin{equation*}
\begin{split}
\boldsymbol{E}\left(\left|X^{i,N}_{t}-Y^{i,N}_{t}\right|^{2}\right)
\leq&C_{T,L}\boldsymbol{E}\left(\int_{0}^{t}\left|X_{s}^{i,N}-\bar{Y}_{s}^{i,N}\right|^{2}\mathrm{~d}s\right)
+C_{T,L}\boldsymbol{E}\left(\int_{0}^{t}\mathcal{W}_{2}^{2}\left(\mu_{s}^{X,N},\bar{\mu}_{s}^{Y,N}\right)\mathrm{~d}s\right).
\end{split}
\end{equation*}
By \eqref{Wpp} and Lemma \ref{XY1}, we get
\begin{equation*}
\begin{split}
\boldsymbol{E}\left(\left|X^{i,N}_{t}-Y^{i,N}_{t}\right|^{2}\right)
\leq&C_{T,L}\int_{0}^{t}\boldsymbol{E}\left(\left|X_{s}^{i,N}-\bar{Y}_{s}^{i,N}\right|^{2}\right)\mathrm{~d}s
\\\leq&C_{T,L}\int_{0}^{t}\boldsymbol{E}\left(\left|X_{s}^{i,N}-Y_{s}^{i,N}\right|^{2}\right)\mathrm{~d}s
+C_{T,L}\int_{0}^{t}\boldsymbol{E}\left(\left|Y_{s}^{i,N}-\bar{Y}_{s}^{i,N}\right|^{2}\right)\mathrm{~d}s
\\\leq& C_{T,L}\int_{0}^{t}\boldsymbol{E}\left(\left|X_{s}^{i,N}-Y_{s}^{i,N}\right|^{2}\right)\mathrm{~d}s
+C_{T,L}\Delta^{2H}.
\end{split}
\end{equation*}
Then applying the Gr\"{o}nwall inequality, we obtain the
desired result.\eproof
\end{proof}

\begin{theorem}
Let Assumption \ref{assp} be satisfied. If $q>2$, then it holds that
\begin{equation*}
\begin{split}
\sup _{i \in\{1, \ldots, N\}} \sup _{t \in[0, T]} \boldsymbol{E}&\left(\left|X_{t}^{i}-Y_{t}^{i,N}\right|^{2}\right) \leq C_{T,H,L}\\&
\times
\left\{\begin{array}{ll}
N^{-1 / 2}+N^{-(q-2)/q}+\Delta^{2H}, & \text { if } d<4 \\
N^{-1 / 2} \log(1+N)+N^{-(q-2)/q}+\Delta^{2H}, & \text { if } d=4\\
N^{-2 / d}+N^{-(q-2)/q}+\Delta^{2H}, & \text { if }d>4,
\end{array}\right.
\end{split}
\end{equation*}
where the constant $C_{T,H,L}> 0$ does not depend on $N$ and $\Delta$.
\end{theorem}
We can easily get this theorem through trigonometric inequality, Theorem \ref{pc1} and Theorem \ref{thm1}.

\section{Numerical Example}\label{part4}
\begin{example}\label{example1}
Consider the following  McKean–Vlasov SDEs driven by fBm
\begin{equation}
\mathrm{~d}X_{t}=\left(X_{t}+\int_{\mathbb{R}}(X_{t}-y)\mu(\mathrm{d}y)\right)\mathrm{~d}t
+\left(\int_{\mathbb{R}}(X_{t}-y)\mu(\mathrm{d}y)\right)\mathrm{~d}B^{H}_{t},
\end{equation}
where initial value $X_{0}$ is a constant.
\end{example}


\begin{figure}[htbp]
 \centering
  \subfigure{\label{fig_c1}
 \includegraphics[width=5.5cm]{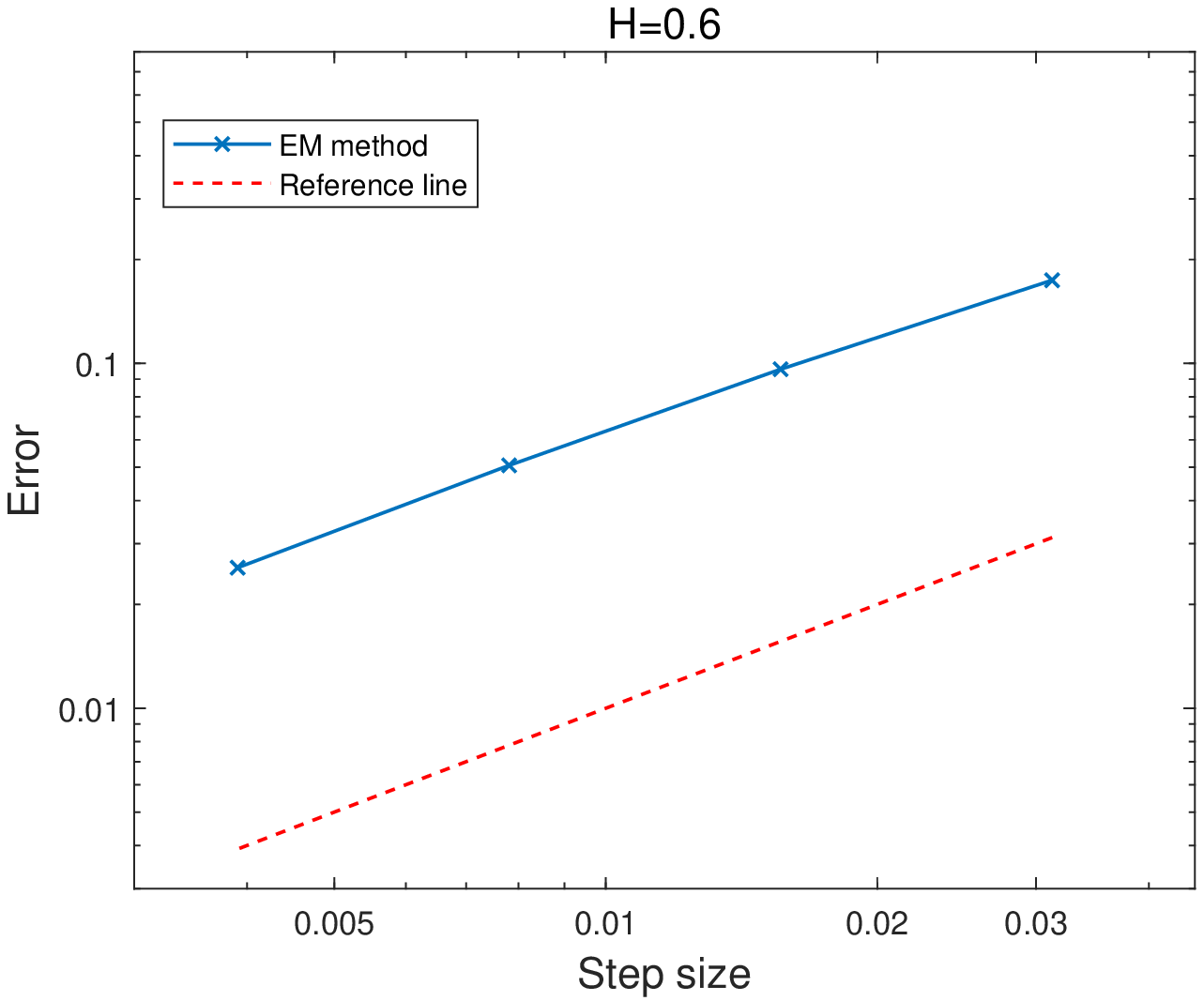}}
 \subfigure{\label{fig_c2}
 \includegraphics[width=5.5cm]{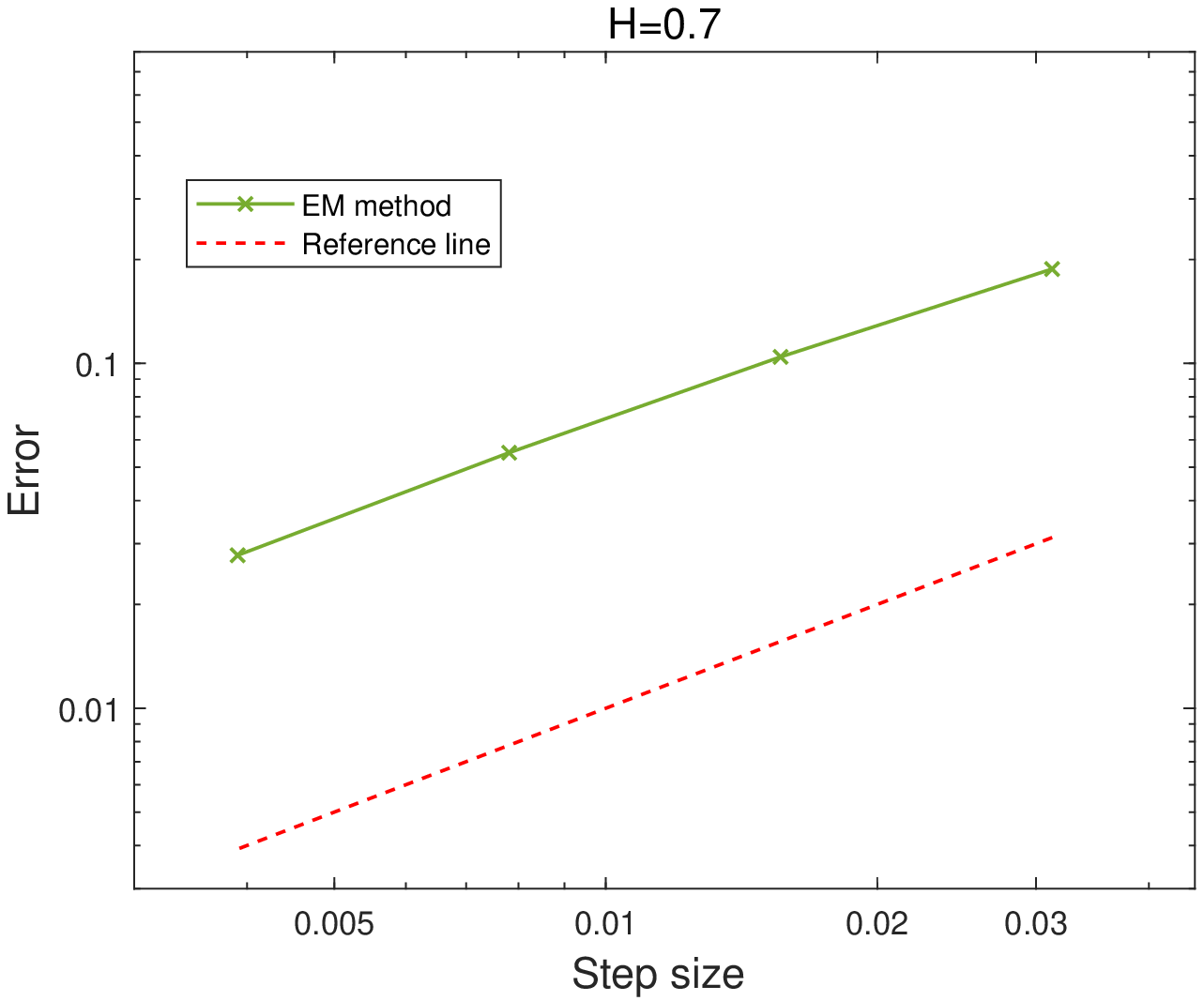}}
\subfigure{
 \label{fig_c3}
 \includegraphics[width=5.5cm]{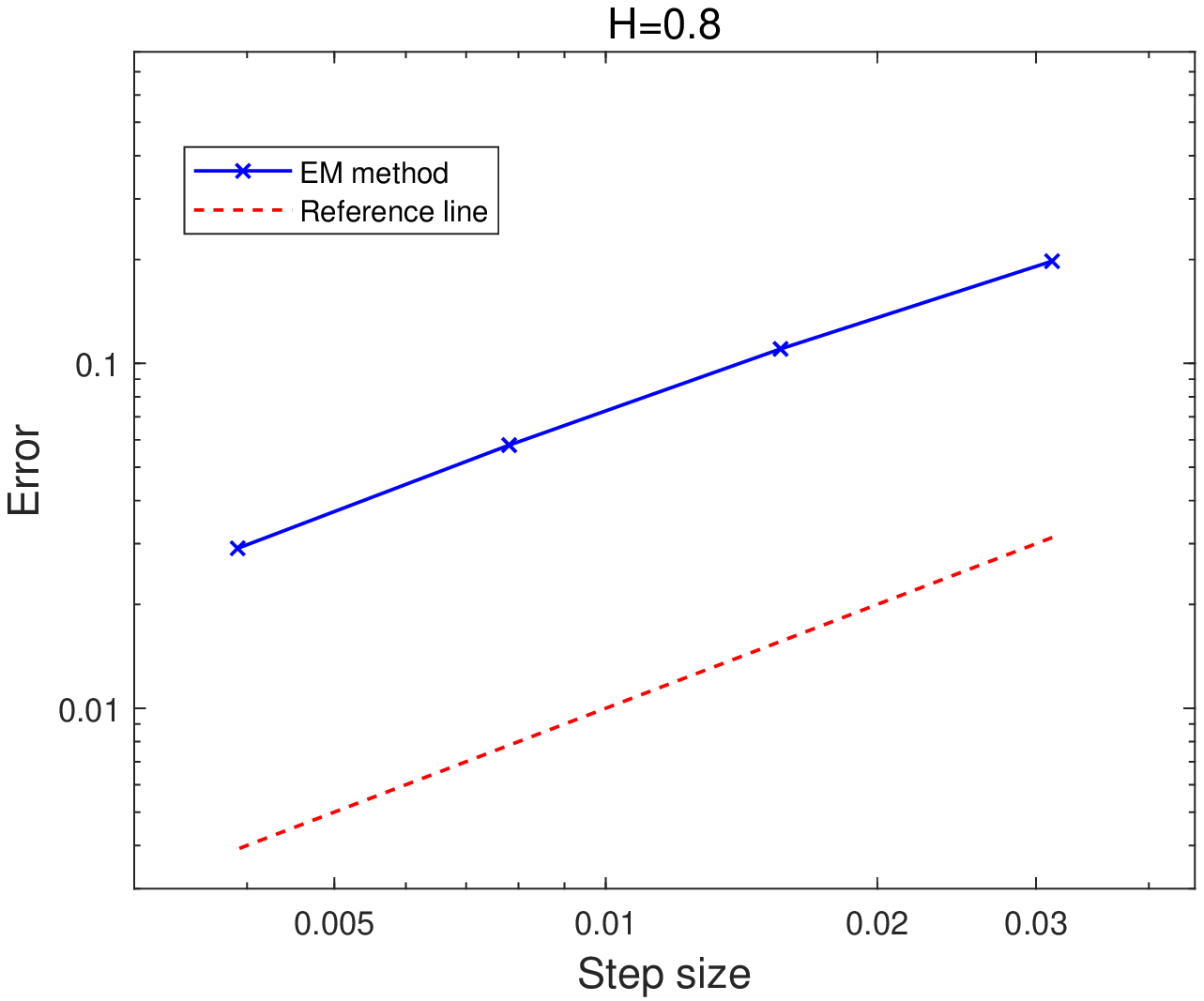}}
 \subfigure{\label{fig_c4}
 \includegraphics[width=5.5cm]{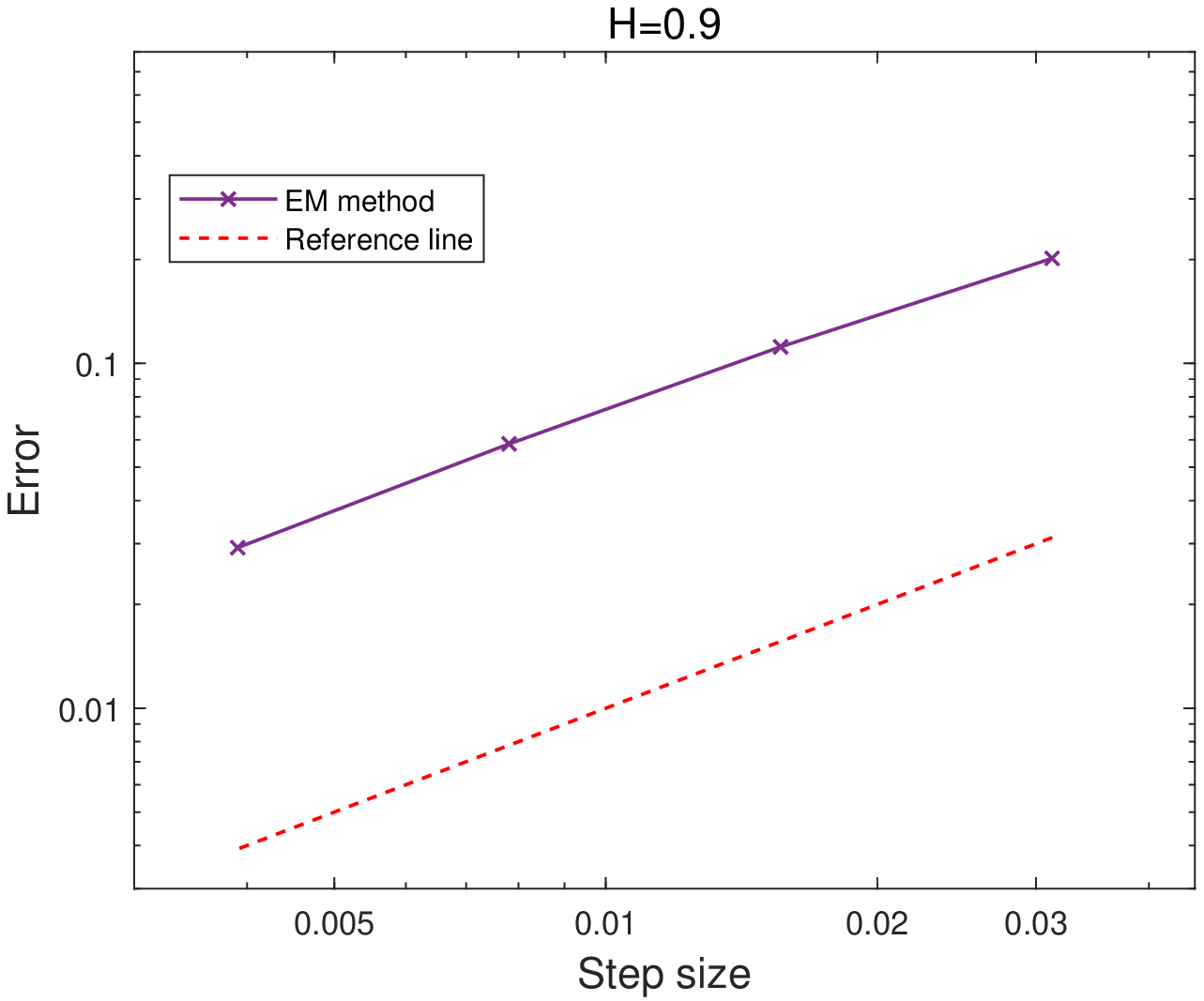}}
\caption{Error in final solution at time $T=1$ as the mean step size decreases for the truncated EM
method applied to  (\ref{example1})  with parameter set initial value $X_0 = 1$ and particle number $U=1000$.}\label{tu}
 \end{figure}

In Figure \ref{tu}, we draw the terminal mean square error at $T=1$ with four time step sizes ($\Delta=2^{-5}, 2^{-6}, 2^{-7}, 2^{-8}$) and use 100 sample paths to obtain the convergence order of $H=0.6$, $H=0.7$, $H=0.8$ and $H=0.9$ respectively. We use the numerical solution with smaller step size $2^{12}$ to replace the exact solution. For the four subfigures in Figure \ref{tu} , the red dotted reference line has a slope of 1, this observation is in line with our theoretical result.

\section*{Acknowledgements}
The corresponding  author is supported by the National Natural Science Foundation of China (No. 11871343). The work of W. Zhan is  supported by University Natural Science Research Project of Anhui (No. KJ2021A0107).

\appendix

\section{Proof of Theorem \ref{pc1}}\label{pcc1}

\begin{proof}
It follows from \eqref{McKeanN} and \eqref{McKeanI} that
\begin{equation*}
X_{t}^{i}-X_{t}^{i,N}=\int_{0}^{t}\left(b\left(X_{s}^{i}, \mathcal{L}_{X_{s}^{i}}\right)-b\left(X_{s}^{i,N}, \mu_{s}^{X,N}\right)\right) \mathrm{~d}s.
\end{equation*}
By \eqref{JJ1}, we obtain
\begin{equation*}
\begin{split}
\boldsymbol{E}\left(\left|X_{t}^{i}-X_{t}^{i,N}\right|^{2}\right)\leq& C_{T,L}\boldsymbol{E}\left(\int_{0}^{t}\left|X_{s}^{i}-X_{s}^{i,N}\right|^{2}\mathrm{~d}s\right)
+C_{T,L}\int_{0}^{t}\boldsymbol{E}\left(\mathcal{W}_{2}^{2}(\mathcal{L}_{X_{s}^{i}},\mu_{s}^{X,N} )\right)\mathrm{~d}s
\\\leq&C_{T,L}\int_{0}^{t}\boldsymbol{E}\left(\left|X_{s}^{i}
-X_{s}^{i,N}\right|^{2}\right)\mathrm{~d}s
+C_{T,L}\int_{0}^{t}\boldsymbol{E}\left(\mathcal{W}_{2}^{2}(\mathcal{L}_{X_{s}^{i}},\mu_{s}^{X} )\right)
\mathrm{~d}s.
\end{split}
\end{equation*}
We know that $\mathcal{W}_{2}^{2}(\mathcal{L}_{X_{s}^{i}},\mu_{s}^{X} )$ is controlled by the Wasserstein distance estimate  in \cite[Theorem 1]{fournier2015}. Therefore,
\begin{equation*}
\begin{split}
\boldsymbol{E}\left(\mathcal{W}_{2}^{2}(\mathcal{L}_{X_{s}^{i}},\mu_{s}^{X} )\right)\leq& CM^{2/q}_{q}(\mathcal{L}_{X_{s}^{i}})\\&
\times\left\{\begin{array}{ll}
N^{-1 / 2}+N^{-(q-2)/q}, & \text { if } d<4, \\
N^{-1 / 2} \log(1+N)+N^{-(q-2)/q}, & \text { if } d=4,\\
N^{-2 / d}+N^{-(q-2)/q}, & \text { if }d>4,
\end{array}\right.
\end{split}
\end{equation*}
where $$M_{q}(\mathcal{L}_{X_{s}^{i}})=\int_{\mathbb{R}^{d}}\left|X_{s}^{i}\right|^{q}\mathcal{L}_{X_{s}^{i}}(\mathrm{~d}X_{s}^{i}).$$
By Theorem \ref{Xbound1}, we note that $M_{q}(\mathcal{L}_{X_{s}^{i}})\leq C$.
Thus,
\begin{equation*}
\begin{split}
\boldsymbol{E}\left(\mathcal{W}_{2}^{2}(\mathcal{L}_{X_{s}^{i}},\mu_{s}^{X} )\right)\leq C
\left\{\begin{array}{ll}
N^{-1 / 2}+N^{-(q-2)/q}, & \text { if } d<4, \\
N^{-1 / 2} \log(1+N)+N^{-(q-2)/q}, & \text { if } d=4,\\
N^{-2 / d}+N^{-(q-2)/q}, & \text { if }d>4,
\end{array}\right.
\end{split}
\end{equation*}
Then, applying the Gr\"{o}nwall inequality completes the proof.\eproof
\end{proof}

\bibliographystyle{plain}
\bibliography{cite}

\end{document}